\documentclass[12pt]{amsart}
\usepackage{amsmath,amssymb,latexsym, amsthm, amscd, mathrsfs, stmaryrd} %, txfonts}

\usepackage[all]{xy}
\usepackage{hyperref}
\usepackage{color}
\newcommand{\nc}{\newcommand}
\nc{\browntext}[1]{\textcolor{brown}{#1}}
\nc{\greentext}[1]{\textcolor{green}{#1}}
\nc{\redtext}[1]{\textcolor{red}{#1}}
\nc{\bluetext}[1]{\textcolor{blue}{#1}}
\nc{\brown}[1]{\browntext{ #1}}
\nc{\green}[1]{\greentext{ #1}}
\nc{\red}[1]{\redtext{ #1}}
\nc{\blue}[1]{\bluetext{ #1}}
\nc{\zb}[1]{\redtext{From zb: #1}}

\newtheorem{thm}{Theorem}  [section]

\newtheorem{lem}[thm]{Lemma}
\newtheorem{prop}[thm]{Proposition}
\newtheorem{example}[thm]{Example}

\theoremstyle{remark}
\newtheorem{rem}[thm]{Remark}

\numberwithin{equation}{section}

\newcommand{\mbf}{\mathbf}

\newcommand{\mrm}{\mathrm}
\newcommand{\A}{\mathcal A}
\newcommand{\cbinom}[2]{\left\{ \begin{matrix} #1\\#2 \end{matrix} \right\}}
\newcommand{\dvev}[1]{{\mathfrak{t}}_{\ev}^{{(#1)}}}
\newcommand{\dv}[1]{{\mathfrak{t}}_{\odd}^{{(#1)}}}
\newcommand{\dvd}[1]{t_{\odd}^{{(#1)}}}
\newcommand{\dvp}[1]{t_{\ev}^{{(#1)}}}
\newcommand{\ev}{\mrm{ev}}

\newcommand{\kk}{h}
\newcommand{\la}{\lambda}
\newcommand{\LR}[2]{\left\llbracket \begin{matrix} #1\\#2 \end{matrix} \right\rrbracket}
\newcommand{\N}{\mathbb N}
\newcommand{\odd}{\mrm{odd}}
\newcommand{\one}{\mathbf 1}

\newcommand{\qbinom}[2]{\begin{bmatrix} #1\\#2 \end{bmatrix} }
\newcommand{\Q}{\mathbb Q}
\newcommand{\sll}{\mathfrak{sl}}
\newcommand{\ttt}{\mathfrak{t}}
\newcommand{\U}{\mbf U}
\newcommand{\Udot}{\dot{\mbf U}}
\newcommand{\UA}{{}_\A{\mbf U}}
\newcommand{\UAdot}{{}_\A{\dot{\mbf U}}}
\newcommand{\Ui}{{\mbf U}^\imath}

\newcommand{\vev}{v^+_{2\la} }
\newcommand{\vodd}{v^+_{2\la+1} }
\newcommand{\vs}{\varsigma}
\newcommand{\Y}{\check{E}}
\newcommand{\Z}{\mathbb Z}

\newcommand{\B}{\mbf V}

\newcommand{\BA}{{}_\A{\B}}
\newcommand{\DA}{{}_\A{\B}'}

\title[Formulae of $\imath$-divided powers]{Formulae of $\imath$-divided powers in ${\mathbf U}_q(\mathfrak{sl}_2)$}

\author[Collin Berman]{Collin Berman}
\author[Weiqiang Wang]{Weiqiang Wang}
\address{ Department of Mathematics\\ University of Virginia\\ Charlottesville, VA 22904}
\email{cmb5nh@virginia.edu (Berman), ww9c@virginia.edu (Wang)}

\keywords{Quantum groups, canonical basis, divided powers, quantum symmetric pairs} 
\subjclass[2010]{Primary 17B10}

\begin{document}

\begin{abstract} 
The existence of the $\imath$-canonical basis (also known as the $\imath$-divided powers) for the coideal subalgebra of the quantum $\mathfrak{sl}_2$ were established by Bao and Wang, with conjectural explicit formulae. In this paper we prove the conjectured formulae of these $\imath$-divided powers. This is achieved by first establishing  closed formulae of the $\imath$-divided powers in basis for quantum $\mathfrak{sl}_2$ and then formulae for the $\imath$-canonical basis in terms of Lusztig's divided powers in each finite-dimensional simple module of quantum $\mathfrak{sl}_2$. These formulae exhibit integrality and positivity properties.
\end{abstract}

\maketitle

%\setcounter{tocdepth}{1} 
%\tableofcontents

%%
%%
\section{Introduction}

\subsection{}

Denote by $\U$ the quantum group of $\sll_2$ over $\Q(q)$ with standard generators $E, F, K^{\pm 1}$. The divided powers of $F$ have a simple expression 
\[
F^{(n)} =F^n/ [n]!, \qquad \text{ where }\quad [n] =(q^n -q^{-n}) /(q-q^{-1}), 
\]
and $\{F^{(n)} \mid n\ge 0\}$ forms the canonical basis for the negative half of $\U$; cf. Lusztig \cite{L93}. Alternatively, the divided powers satisfy and are in turn determined by the simple recursive relation: 
\[
F \cdot F^{(n)}=[n+1] F^{(n+1)}, \qquad F^{(1)} =F.
\] 

\subsection{}

The coideal subalgebra $\Ui$ of $\U$ is a polynomial algebra in one variable $t$, and
$(\U, \Ui)$ forms a quantum symmetric pair of rank one (cf. Koornwinder \cite{K93}; Letzter \cite{Le99} and the references therein). 
Our convention in this paper follows the paper \cite{BW16} by Bao and the second author, which differs from \cite{BW13}. 
Let us take 
\[
t =q^{-1}EK^{-1}+F+K^{-1}.
\]
As first shown in \cite{BW13},
there are two distinct $\Z[q,q^{-1}]$-forms for $\Ui$, respectively; these two integral forms correspond to a choice of parity $\{\ev, \odd\}$ of highest weights of finite-dimensional simple $\U$-modules. 
The existence of $\imath$-canonical basis (also called $\imath$-divided powers in the current rank one case) 
in each integral form of $\Ui$ was established {\em loc. cit.}.  
Explicit formulae for the $\imath$-divided powers
$\dvp{n}$ and $\dvd{n}$, as polynomials in $t$ of degree $n$, were conjectured in \cite{BW13}
with $\dvp{1} =\dvd{1}=t$, see \eqref{def:idp} and \eqref{eq:dv} below.

A useful and alternative viewpoint for the conjectural formulae of the $\imath$-divided powers   
is that the 
$\imath$-divided powers $\dvp{n}$ of even weights satisfy and are in turn determined by the following recursive relations: 
\begin{align}  
\label{eq:ttodd2:dvp} 
\begin{split}
 \dvp{1} =t, \qquad   t \cdot \dvp{2a} &=  [2a+1] \dvp{2a+1},  
\\
t \cdot \dvp{2a+1} &=[2a+2] \dvp{2a+2}+[2a+1] \dvp{2a}, \quad \text{ for } a \ge 0.
\end{split}
\end{align}
The other $\imath$-divided powers $\dvd{n}$ are determined by different recursive relations; see \eqref{eq:ttodd2:dvd}.

There actually exists a family of embeddings of $\Ui$ into $\U$; cf. \cite{Le99} (the family actually appeared in Koornwinder \cite{K93} under the terminology of twisted primitive elements). More precisely, in the notation of \cite{BW16} we have $t \mapsto q^{-1}EK^{-1}+F+\kappa K^{-1}$, for an arbitrarily fixed (bar invariant) element $\kappa \in \Z[q,q^{-1}]$.
Besides the one considered above, another distinguished choice is to choose $\kappa=0$, that is, to choose the polynomial generator of the coideal algebra $\Ui$ to be 
$$\mathfrak{t}=q^{-1}EK^{-1}+F. 
$$
The second author proposed with Huanchen Bao to use the same conjectural
polynomial formulae for the corresponding $\imath$-divided powers, $\dvev{n}$ and $\dv{n}$, but with the parity reversed; that is, $\dvev{n}$ (and $\dv{n}$, respectively) satisfy the same recursive relations for $\dvd{n}$; see \eqref{eq:tt} (and for $\dvp{n}$, respectively; see \eqref{eq:ttodd2}).

We refer the interested reader to \cite{BW13, LW15, BW16} and the references therein for more background and the larger picture (including category $\mathcal O$, flag varieties, canonical bases) where the $\imath$-divided powers fit in. As the $\imath$-divided powers are basic building blocks for (the integral forms of)  quantum symmetric pairs of higher ranks, for future applications it is desirable to better understand the $\imath$-divided powers
 (whose structures are much more involved than the divided powers in quantum groups). 
%In this paper, we shall take these polynomial definitions of the $\imath$-divided powers as a starting point.

%
%
\subsection{} 

In this paper we first
establish compact closed formulae for $\dvev{n}, \dv{n}$, $\dvp{n}$ and $\dvd{n}$ in terms of PBW (and canonical) bases in $\U$ (and in the modified quantum group $\Udot$). 
%The formulae  indicate that the $\imath$divided powers are integral in a suitable sense. 
We then show that the $\imath$-divided powers send the highest weight vector of an arbitrary finite-dimensional simple $\U$-module $L(\la)$ to zero or an $\imath$-canonical basis element of $L(\la)$ (whose existence was established earlier in \cite{BW13, BW16}). Finally from this and the defining property of the $\imath$-canonical basis for $\Ui$, we establish the BW conjectural formulae for the $\imath$-divided powers of $\Ui$ as polynomials in $t$ (or $\ttt$). All formulae exhibit remarkable integrality and positivity properties. The $q\mapsto 1$ limits of the formulae in this paper are also nontrivial and interesting in their own. 

Along the way, we construct and study some interesting non-standard $\Z[q,q^{-1}]$-subalgebras of $\U$, which contain the $\imath$-divided powers and admit several natural (anti-)involutions. These subalgebras appeared in earlier discussions between Bao and the second author.

\subsection{} 

The $\imath$-divided powers as polynomials in $t$ (or $\mathfrak{t}$) are not simply monomials in contrast to the standard divided powers $F^{(n)}$. Consequently, the existence of closed formulae for these $\imath$-divided powers in the non-commutative algebra $\U$ is a rather remarkable fact and can be viewed as a manifestation of the intrinsic nature of the $\imath$-divided powers. 
We discovered the formulae after studying many examples by hand and by Mathematica. 
The inductive proofs of these formulae are elementary though sometimes lengthy, and they are included for the sake of completeness. 
The formulae for $\dvev{n}$ and $\dv{n}$ are easier and will be presented first in Sections~\ref{sec:dv:even} and \ref{sec:dv:odd} (see Theorems~\ref{thm:iDP} and \ref{thm:iDP:odd}), while the formulae for $\dvp{n}$ and $\dvd{n}$ appear in Sections~\ref{sec:dvK:even} and \ref{sec:dvK:odd} (see Theorems~\ref{thm:iDPK} and \ref{thm:iDPK:odd}). 

These $\imath$-divided power formulae are subsequently reformulated in the modified quantum group $\Udot$; they can be seen as a positive integral (i.e., $\N[q,q^{-1}]$-) linear combination of the canonical basis of $\Udot$ (cf. \cite[25.3]{L93}) with coefficients expressed in terms of $q$-binomial and $q^2$-binomial coefficients. We note a consistent phenomenon that the $q^2$-binomial coefficients also arise naturally in formulae for the $\imath$-canonical basis elements of the $\imath$-Schur algebra or rank one (see \cite[Theorem~7.6]{LW15}). 

The explicit expansion formulae of the $\imath$-divided powers in $\U$ allow us to obtain their precise images on the highest weight vector of any finite-dimensional simple $\U$-module $L(\la)$. We show that the nonzero images form a basis for $L(\la)$, whose transition matrix to the usual canonical basis of $L(\la)$ is uni-triangular with non-diagonal entries in $q^{-1} \Z_{\ge 0} [q^{-1}]$, and hence they are by definition the $\imath$-canonical basis of $\Ui$. (There is a subtle difference between the cases for $t$ and for $\ttt$ though.) From these the BW conjecture that the $\imath$-divided powers form the $\imath$-canonical basis of $\Ui$ follows; see Theorems~\ref{thm:iCB:ev}, \ref{thm:iCB:odd}, \ref{thm:iCB:evK} and \ref{thm:iCB:oddK}  for the cases of $\dvev{n}, \dv{n}$, $\dvp{n}$ and $\dvd{n}$, respectively. 

It is also possible to establish the conjectural $\imath$-divided power formulae in \cite{BW13} in the geometric framework of \cite{LW15}. The two choices of the parameter $\kappa$ (with $\kappa=1$/$0$) are distinguished as they correspond to type B/D flag varieties and Kazhdan-Lusztig theories. The cases with more general parameter $\kappa$ are also interesting but much more challenging, and we shall return to this in a separate occasion. 
The study of $\imath$-divided powers and more generally of $\imath$-canonical bases has opened up new and fruitful interactions with $q$-combinatorics among others, which will be further pursued elsewhere.

\vspace{.3cm}

{\bf Acknowledgement.} 
This project deals with some computational 
aspects of the program on canonical bases arising from quantum symmetric pairs initiated by Huanchen Bao and WW.  We thank a referee for providing the reference \cite{K93}, and WW thanks Huanchen for his insightful collaboration. 
The research of WW  and the undergraduate research of CB are partially supported by a grant from National Science Foundation. 
The senior author acknowledges the great skill of the junior author with Mathematica, which has played an important role in helping produce a large set of examples.

\section{Formulae for $\imath$-divided powers $\dvev{n}$ and $\imath$-canonical basis on $L(2\la)$}
  \label{sec:dv:even}

\subsection{The quantum $\sll_2$}

Let $\Q(q)$ be the field of rational functions in a variable $q$. Let 
\[
\A=\Z[q,q^{-1}].
\] 
Let $\mathbb N =\{0,1, 2, \ldots\}$. The $q$-integers and $q$-binomial coefficients 
are defined as follows: for $m \in \mathbb{Z}$ and $b \in \mathbb N$,
\begin{align*}
 [m] = \frac{q^{m}-q^{-m}}{q- q^{-1}},
\quad
[b]! =[1] [2] \cdots [b],
\quad
\begin{bmatrix}
m\\b
\end{bmatrix}
&=\prod_{1\leq i\leq b} \frac{q^{m-i+1}- q^{-(m-i+1)}}{q^i- q^{-i}}.
\end{align*}
We shall need the following additional notations: for $m \in \mathbb{Z}$ and $b \in \mathbb N$, 
\[
[m]_{q^2} =\frac{q^{2m}-q^{-2m}}{q^2- q^{-2}},
\qquad
\begin{bmatrix}
m\\ b
\end{bmatrix}_{q^2} =\prod_{1\leq i\leq b} \frac{q^{2(m-i+1)}- q^{-2(m-i+1)}}{q^{2i}- q^{-2i}}.
\]

Recall  the quantum group $\U=\U_q(\mathfrak{sl}_2)$ is the $\Q(q)$-algebra generated by $F, E, K, K^{-1}$, subject to the relations:  $KK^{-1}=K^{-1}K=1$, and 
\begin{align}
 \label{sl2}
EF -FE =\frac{K-K^{-1}}{q-q^{-1}},
\quad
K E =q^2 E K,
\quad
K F=q^{-2} FK.
\end{align}
%Denote $\qbinom{K;a}{n} =\prod_{i=1}^n \frac{q^{a} K -q^{-a}K^{-1} }{q^{i} -q^{-i}}$, for $n\in \N$ and $a\in \Z$.
Let $\UA$ be the $\A$-subalgebra of $\U$ generated by $E^{(n)}, F^{(n)},  K^{\pm 1}$ and $\prod_{i=1}^n \frac{q^{a+i} K -q^{-a-i}K^{-1} }{q^{i} -q^{-i}}$, 
for all $n \ge 1$ and $a\in \Z$. 
There is an anti-involution $\vs$ of the $\Q$-algebra $\U$:
\begin{align}
 \label{eq:vs}
\vs: \U \longrightarrow \U,
\qquad E\mapsto E, \quad F\mapsto F, \quad K\mapsto K, \quad q\mapsto q^{-1}.
\end{align}

Denote by $\Udot$ the modified quantum group of $\sll_2$ \cite{L93}, which is the $\Q(q)$-algebra generated by $E\one_\la, F\one_\la$ and the  idempotents $\one_\la$, for $\la\in \Z$, which satisfy the relations 
\begin{align*}
\one_\la \one_\mu =\delta_{\la,\mu} \one_\la, 
&\qquad
E\one_\la =\one_{\la+2} E \one_{\la} =\one_{\la+2} E,
\\
F\one_\la  =\one_{\la-2} F \one_\la=\one_{\la-2} F,
&\qquad 
EF\one_\la -FE \one_\la =[\la].
\end{align*}
Let $\UAdot$ be the $\A$-subalgebra of $\Udot$  generated by $E^{(n)} \one_\la, F^{(n)} \one_\la, \one_\la$, for all $n\ge 0$ and $\la\in \Z$.
There is a natural left action of $\U$ on $\Udot$ such that $K \one_\la =q^\la \one_\la$. 
Denote by 
\[
\UAdot_\ev =\bigoplus_{\la\in \Z} \, \UAdot \one_{2\la},
\qquad
\UAdot_\odd =\bigoplus_{\la\in \Z}\, \UAdot \one_{2\la -1}. 
\]
We have $\UAdot =\UAdot_\ev  \oplus \UAdot_\odd$. By a base change we define  $\Udot_\ev$ and $\Udot_\odd$ accordingly so that $\Udot =\Udot_\ev  \oplus \Udot_\odd$.
\subsection{The $\imath$-divided powers $\dvev{n}$} 

Set 
\begin{equation}  \label{eq:t}
\Y :=q^{-1}EK^{-1}, 
\qquad
\kk:=\frac{K^{-2}-1}{q^2-1}, \qquad \mathfrak{t}:=\Y +F.
\end{equation}
For $a\in \N$,  Bao and the second author  \cite{BW13} proposed 
the following formulae for $\imath$-divided powers  (the terminology of $\imath$-divided power is more recent and did not appear {\em loc. cit.}):
\begin{align}
\label{def:idp}
\begin{split}
\dvev{2a} & = {\mathfrak{t}(\mathfrak{t} - [-2a+2])(\mathfrak{t}-[-2a+4]) \cdots (\mathfrak{t}-[2a-4]) (\mathfrak{t} - [2a-2]) \over [2a]!},\\
\dvev{2a+1} & = {  (\mathfrak{t} - [-2a])(\mathfrak{t}-[-2a+2]) \cdots (\mathfrak{t}-[2a-2]) (\mathfrak{t} - [2a])  \over [2a+1]!}.
\end{split}
\end{align}
(Actually $\mathfrak{t}$ is replaced by $t=\Y+F+K^{-1}$ up to an involution of $\U$ in \cite[Conjecture~4.13]{BW13}; see Sections~\ref{sec:dvK:odd} below. But Bao and the second author suggested to use the same formulae in the current setting for $t$).

Note $\dvev{0}=1$ and $\dvev{1} =\mathfrak{t}$, $\dvev{2} =\mathfrak{t}^2/[2]$, and $\dvev{3} = \mathfrak{t}(\mathfrak{t}^2-[2]^2)/[3]!$. 
The $\imath$-divided powers $\dvev{n}$ satisfy and are in turn determined by the  following recursive relations: 
\begin{align}   
  \label{eq:tt}
\begin{split}
\mathfrak{t} \cdot \dvev{2a-1} &=[2a] \dvev{2a},
\\
\mathfrak{t} \cdot \dvev{2a} &=  [2a+1] \dvev{2a+1} +[2a] \dvev{2a-1}, \quad \text{ for } a\ge 1.
\end{split}
\end{align}

\subsection{The algebra $\BA$} 

Define, for $a\in \Z, n\ge 0$, 
\begin{equation}  \label{kbinom}
\qbinom{\kk;a}{n} =\prod_{i=1}^n \frac{q^{4a+4i-4} K^{-2} -1}{q^{4i} -1},
\qquad 
[\kk;a]= \qbinom{\kk;a}{1}.
\end{equation}
Note that 
$\kk=q[2]\, [\kk;0].$ 
While $\qbinom{\kk;a}{n}$ does not lie in $\UA$ in general,  
what is crucial for us is the fact that, for $\la \in \Z$, 
\begin{equation}
  \label{k=qbinom}
\qbinom{\kk;a}{n} \one_{2\la} = q^{2n(a-1-\la)} \qbinom{a-1-\la+n}{n}_{q^2} \one_{2\la} \in \UAdot_\ev.
\end{equation}
It follows from \eqref{sl2} that, for $n\ge 0$ and $a\in \Z$,  
\begin{align}
  \label{FEk}
F \Y  -q^{-2} \Y  F =\kk,
\qquad
\qbinom{\kk;a}{n} F = F \qbinom{\kk;a+1}{n},
\qquad
\qbinom{\kk;a}{n} \Y  =\Y  \qbinom{\kk;a-1}{n}.
\end{align}

%We have
%\[
% \qbinom{\kk;a-1}{n}  = \qbinom{\kk;a}{n} - \qbinom{\kk;a}{n-1}.
%\]

Let 
\[
\Y^{(n)} =\Y^n/[n]!, \quad
F^{(n)} =F^n/[n]!,\quad \text{ for } n\ge 1.
\] 
It is understood that $\Y^{(n)}=0$ for $n<0$ and $\Y^{(0)}=1$.
Denote by $\B$ the $\Q(q)$-subalgebra of $\U$ generated by $\Y, F$ and $K^{-1}$.  
Denote by $\BA$ the $\A$-subalgebra of $\B$ with $1$ generated by $\Y^{(n)}, F^{(n)}, K^{-1}$ and $\qbinom{\kk;a}{n}$, 
for all $n \ge 1$ and $a\in \Z$. 

\begin{lem}
  \label{dpY}
For $n\in \N$, we have 
\[
\Y^{(n)}= q^{-n^2} E^{(n)} K^{-n}.
\]
\end{lem}

\begin{proof}
Follows by induction on $n$, using \eqref{sl2} and \eqref{eq:t}.
\end{proof}

\begin{lem}
  \label{lem:anti}
The anti-involution $\vs$ on $\U$ restricts to an anti-involution of the $\Q$-algebra $\B$ sending 
\[
F\mapsto F,\quad \Y \mapsto \Y, \quad  K^{-1} \mapsto K^{-1},  \quad  q \mapsto q^{-1}.
\] 
Moreover,
$\vs$ sends 
\[
\kk\mapsto -q^{2} \kk, 
\quad
\dvev{n} \mapsto \dvev{n}, 
\quad
\qbinom{\kk;a}{n}\mapsto (-1)^n q^{2n(n+1)} \qbinom{\kk;1-a-n}{n}, \; \forall a\in\Z, n\in\N.
\] 
\end{lem}

\begin{proof}
Using \eqref{eq:vs} we compute
$\vs(\Y) =\vs(q^{-1}EK^{-1}) =qK^{-1}E =\Y.$
Hence $\vs$ restricts to an anti-involution of the $\Q$-algebra $\B$.
It is straightforward to verify that $\vs(\kk) =-q^{2} \kk$, $\vs(\dvev{n}) = \dvev{n}$.
Finally, we have
\begin{align*}
\vs\left( \qbinom{\kk;a}{n}\right)
&= \prod_{i=1}^n \frac{q^{-(4a+4i-4)} K^{-2} -1}{q^{-4i} -1}
\\
&= \prod_{i=1}^n (-q^{4i}) \prod_{i=1}^n \frac{q^{ 4(1-a-n) +4i-4} K^{-2} -1}{q^{4i} -1}
=  (-1)^n q^{2n(n+1)} \qbinom{\kk;1-a-n}{n}.
\end{align*}
This proves the lemma. 
\end{proof}

\begin{rem}
An involution $\varpi$ on the $\Q$-algebra $\U$,
\[
\varpi: E \mapsto q^{-1}FK, \quad F\mapsto q^{-1}EK^{-1},
\quad K^{-1} \mapsto K^{-1},
\quad q\mapsto q^{-1}, 
\]
%$E \mapsto q^{-1}FK$, $F\mapsto q^{-1}EK^{-1}$, $K^{-1} \mapsto K^{-1}$, $q\mapsto q^{-1}$ 
restricts to an involution $\varpi$ on the $\Q$-algebra $\B$, switching $\Y \leftrightarrow F.$ 
We can use $\varpi$ in place of the anti-involution $\vs$ for applications below. 

There is yet another anti-involution of the $\Q(q)$-algebra $\U$ (and $\B$) which fixes
$K^{-1}$ and switches $F$ and $\Y.$ We do not need this anti-involution below.
\end{rem}

\begin{lem}
The following formula holds for $n\ge 0$:  
\begin{align}   \label{FYn}
%F^n \Y &=q^{-2n} \Y F^n + (q^2-1)^{-1} [n] F^{n-1} \big(  q^{n-1} K^{-2} - q^{1-n}  \big),\\
%F^{(n)} \Y  &=q^{-2n} \Y F^{(n)} + \frac{q^{3-3n} K^{-2} - q^{1-n} }{q^2-1} F^{(n-1)}  \\
F \Y^{(n)} &=q^{-2n} \Y^{(n)} F +  \Y^{(n-1)} \frac{q^{3-3n} K^{-2} - q^{1-n} }{q^2-1}.
\end{align} 
\end{lem}

\begin{proof}
We shall prove the following equivalent formula  by induction on $n$:
\[
F\Y^n = q^{-2n} \Y^n F + (q^2-1)^{-1} [n] \Y^{n-1} \big( q^{3-3n} K^{-2} - q^{1-n}  \big).
\]
The base case when $n=1$ is covered by \eqref{FEk}. 
Assume the formula is proved for $F\Y^n$.
Then by inductive assumption we have
\begin{align*}
F\Y^{n+1} &=F\Y^n \Y = q^{-2n} \Y^n F\Y + (q^2-1)^{-1} [n] \Y^{n-1} \big( q^{3-3n} K^{-2} - q^{1-n}  \big) \Y \\
&=q^{-2n} \Y^n (q^{-2} \Y F + (q^2-1)^{-1} \big( K^{-2} - 1\big)  ) + (q^2-1)^{-1} [n] \Y^{n} \big( q^{-1-3n} K^{-2} - q^{1-n}  \big) \\
&=q^{-2(n+1)} \Y^{n+1} F + (q^2-1)^{-1} [n+1 ] \Y^n \big( q^{-3n} K^{-2} - q^{-n}\big). 
\end{align*}
The lemma is proved. 
\end{proof}
\subsection{The $\Y \kk F$-formula for $\dvev{n}$} 

Recall $\qbinom{\kk;a}{n}$ from \eqref{kbinom}. 

\begin{thm}  \label{thm:iDP}
For $m\ge 1$, we have
\begin{align}
\dvev{2m} &= \sum_{c=0}^m \sum_{a=0}^{2m-2c} q^{\binom{2c}{2} -a(2m-2c-a)} \Y^{(a)}  \qbinom{\kk;1-m}{c}  F^{(2m-2c-a)}, 
\label{t2m}
\\
\dvev{2m-1} &=  \sum_{c=0}^{m-1} \sum_{a=0}^{2m-1-2c} q^{\binom{2c+1}{2} -a(2m-1-2c-a)} \Y^{(a)}  \qbinom{\kk; 1-m}{c}  F^{(2m-1-2c-a)}. 
\label{t2m-1}
\end{align}
\end{thm}
The proof of Theorem~\ref{thm:iDP} will be given in \S\ref{sec:proof:iDP}. 
It follows from Theorem~\ref{thm:iDP} that $\dvev{n} \in \BA$ for all $n$.

For $n\in \N$, we denote
\begin{align}  \label{eq:divB}
b^{(n)}  = \sum_{a=0}^n  q^{-a(n-a)}  \Y^{(a)}F^{(n-a)}.
% \qquad \text{ and } \quad b_n = [n]! b^{(n)}.
\end{align}
Note  that the summation over $a$  for $c=0$ in \eqref{t2m}--\eqref{t2m-1} is equal to $b^{(2m)}$ and $b^{(2m-1)}$, respectively. 

\begin{example}
We have the following examples of $\dvev{n}$, for $2\le n\le 5$:
\begin{align*}
\dvev{2} &= b^{(2)} + q [\kk;0],
\\
\dvev{3}  &   = b^{(3)} + q^3[\kk;-1]F+ q^3 \Y [\kk;-1],
\\
\dvev{4} &   = b^{(4)}  +  q\Y^{(2)}  [\kk;-1] +  q  [\kk;-1] F^{(2)} + \Y [\kk; -1] F + q^{6}  \qbinom{\kk;-1}{2},
\\ %%
\dvev{5}  &   = b^{(5)}   + q^3 \Y^{(3)} [\kk;-2]    + q^3 [\kk;-2]F^{(3)} + q \Y^{(2)}[\kk;-2]F  + q \Y [\kk;-2]F^{(2)} \\
&\qquad + q^{10} \Y \qbinom{\kk;-2}{2}  + q^{10} \qbinom{\kk;-2}{2}   F. 
%\\ %%
%t^{(6)}  &=
%b^{(6)} + q \Y^{(4)}[\kk;-2] + q [\kk;-2]F^{(4)} + q^{-2} \Y^{(3)}[\kk;-2]F + q^{-2} \Y [\kk;-2]F^{(3)} 
%+ q^{-3} \Y^{(2)}[\kk;-2]F^{(2)} \\
%&\quad + q^{6} \Y^{(2)} \qbinom{\kk;-2}{2} + q^{6} \qbinom{\kk;-2}{2} F^{(2)}  + q^{5} \Y\qbinom{\kk;-2}{2} F 
%\quad + q^{15} \qbinom{\kk;-2}{3}.
\end{align*}
\end{example}

\subsection{Some reformulations}

Let us reformulate the formulae in Theorem~\ref{thm:iDP} in two different forms.

First we have the following $Fk\Y$-expansion formulae, which can be obtained 
from the formulae in Theorem~\ref{thm:iDP} by applying the anti-involution $\vs$.
\begin{prop}  \label{iDP:FkY}
For $m\ge 1$, we have
\begin{align*}
\dvev{2m} &= \sum_{c=0}^m \sum_{a=0}^{2m-2c} (-1)^c q^{3c+ a(2m-2c-a)} F^{(a)}  \qbinom{\kk;m-c}{c}  \Y^{(2m-2c-a)}, 
\\
\dvev{2m-1} &=  \sum_{c=0}^{m-1} \sum_{a=0}^{2m-1-2c} (-1)^c q^{c+ a(2m-1-2c-a)} F^{(a)}  \qbinom{\kk; m-c}{c}  \Y^{(2m-1-2c-a)}. 
\end{align*}
\end{prop}

\begin{proof}
Let us derive the first formula only from Theorem~\ref{thm:iDP}, and skip the entirely similar proof of the second one.
% (and other similar $Fk\Y$-expansion formulae later on). 
We apply the anti-involution $\vs$ to both sides of \eqref{t2m}. 
The LHS gives us $\vs (\dvev{2m}) =\dvev{2m}$  by Lemma~\ref{lem:anti}.
By using Lemma~\ref{lem:anti} again the RHS gives us 
\begin{align*}
\vs \big(\text{RHS }\eqref{t2m} \big)
& =\sum_{c=0}^m \sum_{a=0}^{2m-2c} q^{-\binom{2c}{2} +a(2m-2c-a)}   F^{(2m-2c-a)}   
(-1)^c q^{2c(c+1)} \qbinom{\kk;m-c}{c}  \Y^{(a)}  
\\
&= \sum_{c=0}^m \sum_{a=0}^{2m-2c} (-1)^c q^{3c+ a(2m-2c-a)}   F^{(2m-2c-a)}  \qbinom{\kk;m-c}{c}  \Y^{(a)}.
%corrected: 1-m >>> m-c, twice
\end{align*}
By a change of variable $a \mapsto 2m-2c-a$, the RHS of the above equation is equal to the one given in the proposition. 
\end{proof}

Next, we have the following modified reformulation of Theorem~\ref{thm:iDP}.

\begin{prop}
  \label{prop:iDPdot}
For $m\ge 1$ and $\la \in \Z$, we have
\begin{align}
\dvev{2m} \one_{2\la}
&= \sum_{c=0}^m \sum_{a=0}^{2m-2c} q^{2(a+c)(m-a-\la)-2ac-\binom{2c+1}{2}} \qbinom{m-c-a-\la}{c}_{q^2}
 E^{(a)}  F^{(2m-2c-a)}\one_{2\la}, 
\label{t2mdot}
\\
\dvev{2m-1} \one_{2\la}
&=  \sum_{c=0}^{m-1} \sum_{a=0}^{2m-1-2c}  
\label{t2m-1dot}\\
&  
\quad q^{2(a+c)(m-a-\la)-2ac-a-\binom{2c+1}{2}}  \qbinom{m-c-a-\la-1}{c}_{q^2}  E^{(a)}  F^{(2m-1-2c-a)}\one_{2\la}. 
\notag
\end{align}
In particular, we have $\dvev{n}\one_{2\la} \in \UAdot_\ev$, for all $n\in \N$. 
\end{prop}

\begin{proof}
Note $F^{(2m-2c-a)}\one_{2\la} = \one_{2(\la-2m+2c+a)} F^{(2m-2c-a)}$; cf. \eqref{t2m}. Using \eqref{k=qbinom} and the formula $\Y^{(a)}= q^{-a^2} E^{(a)} K^{-a}$ from Lemma~\ref{dpY}, we compute
\begin{align*}
&q^{\binom{2c}{2} -a(2m-2c-a)} \Y^{(a)} \qbinom{\kk;1-m}{c}  \one_{2(\la-2m+2c+a)}
\\
&= q^{\binom{2c}{2} -a(2m-2c-a)}  q^{-a^2 -2a (\la -2m +2c +a)} q^{2c(m-2c-a-\la)} \qbinom{m-c-a-\la}{c}_{q^2} E^{(a)} \one_{2(\la-2m+2c+a)}
 \\
&= q^{2(a+c)(m-a-\la)-2ac-\binom{2c+1}{2}} \qbinom{m-c-a-\la}{c}_{q^2} E^{(a)} \one_{2(\la-2m+2c+a)}.
\end{align*}
This proves \eqref{t2mdot}.
%\begin{align*}
%\dvev{2m} \one_{2\la}
%&= \sum_{c=0}^m \sum_{a=0}^{2m-2c} q^{-\binom{2c+1}{2} -a(2m-a) +2c(m-\la)} \qbinom{m-c-a-\la}{c}_{q^2}
% \Y^{(a)}  F^{(2m-2c-a)}\one_{2\la}.
%\end{align*}

Similarly, \eqref{t2m-1dot} follows from \eqref{t2m-1} and \eqref{k=qbinom}
by the following computation:
\begin{align*}
&q^{\binom{2c+1}{2} -a(2m-1-2c-a)}  \Y^{(a)} \qbinom{\kk;1-m}{c}  \one_{2(\la-2m+2c+a+1)}
\\
&= q^{\binom{2c+1}{2} -a(2m-1-2c-a)}  q^{-a^2 -2a (\la -2m +2c +a+1)}  
\cdot
\\
&\qquad\qquad\qquad q^{2c(m-2c-a-\la-1)} \qbinom{m-c-a-\la-1}{c}_{q^2} E^{(a)}  \one_{2(\la-2m+2c+a+1)}
 \\
&= q^{2(a+c)(m-a-\la)-2ac-a-\binom{2c+1}{2}}  \qbinom{m-c-a-\la-1}{c}_{q^2} E^{(a)}  \one_{2(\la-2m+2c+a+1)}.
\end{align*}
The proposition is proved. 
\end{proof}

\begin{rem}
Using \cite[23.1.3]{L93} one can rewrite explicitly each summand in the formulae 
\eqref{t2mdot}--\eqref{t2m-1dot} as a $\N[q,q^{-1}]$-linear combination of 
the canonical basis of $\UAdot$  \cite[25.3]{L93} 
(the coefficients are $q$-binomial coefficients up to suitable $q$-powers).
It follows that $\dvev{n} \one_{2\la}$ has a positive integral (i.e. $\N[q,q^{-1}]$-)expansion with respect to the canonical basis of $\UAdot$.  The same remark applies to other 
$\imath$-divided powers below.
\end{rem}

\subsection{The $\imath$-canonical basis for simple $\U$-modules $L(2\la)$}
  \label{sec:iCB:ev}
  
Denote by $L(\mu)$ the  %$(\mu+1)$-dimensional 
simple $\U$-module of highest weight $\mu \in \N$, with highest weight vector $v^+_\mu$.
Then $L(\mu)$ admits a canonical basis $\{ F^{(a)} v^+_\mu\mid 0\le a \le \mu\}$. 
Following \cite{BW13,BW16}, there exists a new bar involution $\psi_\imath$ on $L(\mu)$, which, in our current rank one setting, can be defined simply by requiring $\dvev{n} v^+_\mu$ to be $\psi_\imath$-invariant for all $n$. 
As a very special case of the general results in \cite[Corollary~F]{BW16} (also cf. \cite{BW13}), we know that
 the {\em $\imath$-canonical basis} $\{b^{\mu}_{a} \}_{0\le a \le \mu}$ of $L(\mu)$ exists and is characterized by the following 2 properties: 

(1) $b^{\mu}_{a}$ is $\psi_\imath$-invariant; \qquad 

(2) $b^{\mu}_{a} \in F^{(a)} v^+_\mu + \sum_{0\le r\le \mu} q^{-1} \Z[q^{-1}] F^{(r)} v^+_\mu$. 

\noindent (It is proved that the summation in (2) only runs over $0\le r <a$.)
%We shall simply take this as our definition of the $\imath$-canonical basis of $L(\mu)$ in this paper.

 For $c\in \N, m\in\Z$, we define
\begin{align}   \label{cbinom}
  \cbinom{m}{c} &= \prod_{i=1}^c \frac{q^{4(m+i-1)} -1}{q^{-4i}-1},
  \qquad \cbinom{m}{0}=1.
\end{align}
 Note $\cbinom{m}{c} \in \N[q,q^{-1}]$ is $q^2$-binomial coefficients up to some $q$-powers.

We shall establish the following basic result, similar to \cite[Conjecture~4.13]{BW13}, which was also conjectured by Bao and the second author (in private).
 \begin{thm}  \label{thm:iCB:ev}
   $\quad$
\begin{enumerate}
\item
For each $\la \in \N$, the set $\{\dvev{n} \vev \mid 0\le n \le 2\la \}$ forms the $\imath$-canonical basis for $L(2\la)$. 
 Moreover, $\dvev{n} \vev =0$ for $n>2\la$.
 \item
The set $\{\dvev{n} \mid n \in \N \}$ forms the $\imath$-canonical basis for ${\U}^\imath$. 
\end{enumerate}
 \end{thm}
 
 \begin{proof}
Let $\la, m \in \N$ with $m\ge 1$.  We compute by Proposition~\ref{iDP:FkY} that 
\begin{align} 
\dvev{2m} \vev 
%&= \sum_{c=0}^m (-1)^c q^{3c} F^{(2m-2c)}  \qbinom{\kk;m-c}{c} \vev \notag \\
& = \sum_{c=0}^m q^{-2c^2+c} \cbinom{m-\la-c}{c} F^{(2m-2c)}  \vev;
  \label{dvev:2m-mod} \\
\dvev{2m-1} \vev 
%&= \sum_{c=0}^{m-1} (-1)^c q^{c} F^{(2m-1-2c)}  \qbinom{\kk;m-c}{c} \vev \notag \\
&= \sum_{c=0}^{m-1} q^{-2c^2-c} \cbinom{m-\la-c}{c} F^{(2m-1-2c)}  \vev.
  \label{dvev:2m-1-mod}
\end{align}
We observe that 
\begin{align}   
 F^{(2m-2c)}  \vev = F^{(2m-1-2c)}  \vev =0, & \quad \text{ if } c<m-\la; 
  \label{ev=0} \\
 \cbinom{m-\la-c}{c} =0, & \quad \text{ if } c\ge m-\la \text{ and } m \ge \la+1.
\label{cbin=0}
\end{align}

It follows from \eqref{cbinom} that, for $m\le \la$ and $c\ge 1$,
\begin{align}   \label{cbin2}
  \cbinom{m-\la-c}{c} \in \N[q^{-1}],  
  \qquad \text{ and }
  -2c^2 \pm c <0.
\end{align}

Using \eqref{dvev:2m-mod}--\eqref{dvev:2m-1-mod} and \eqref{ev=0}--\eqref{cbin=0}, we conclude that $\dvev{2m} \vev =\dvev{2m-1} \vev =0$, for $m\ge \la+1$; using in addition \eqref{cbin2} we conclude that, for $m\le \la$, 
 \begin{align*}
 \dvev{2m} \vev  & \in F^{(2m)}  \vev + \sum_{c\ge 1}q^{-1} \N[q^{-1}] F^{(2m-2c)} \vev,
 \\
 \dvev{2m-1} \vev  & \in F^{(2m-1)}  \vev + \sum_{c\ge 1}q^{-1} \N[q^{-1}] F^{(2m-1-2c)} \vev.
 \end{align*}
 Therefore, the first statement follows by the characterization properties (1)--(2) of the $\imath$-canonical basis for $L(2\la)$.
 
 The second statement follows now from the definition of  the $\imath$-canonical basis for ${\U}^\imath$ using
the projective system $\{ L(2\la) \}_{\la\in \N}$ with $\Ui$-homomorphisms $L(2\la+2) \rightarrow L(2\la)$; cf. \cite[\S6]{BW16} (also cf. \cite[\S4]{BW13}). 
 \end{proof}

\begin{rem}
The $\imath$-canonical basis  was formulated in \cite{BW16} in a modified form of a coideal algebra in the setting of general quantum symmetric pairs.  In the current setting $\Ui$ has no Cartan part and the modified form depends only on the parity $\{ \ev, \odd\}$. We choose not to specify them as they are clear from the parity of the highest weight of the $\U$-modules which they act on. Similar remarks apply to the variants considered in later sections. 
\end{rem} 

Note by \eqref{dvev:2m-mod}-\eqref{dvev:2m-1-mod} that the $\imath$-canonical basis for $L(2\la)$ relative to the usual canonical basis for $L(2\la)$ is uni-triangular with non-diagonal entries in $q^{-1}\N[q^{-1}]$.

\subsection{Proof of Theorem~\ref{thm:iDP} }
  \label{sec:proof:iDP}
We prove the formulae for $\dvev{n}$ by using the recursive relations \eqref{eq:tt} 
and induction on $n$. The base cases for $n=1,2$ are clear.
The induction is carried out in 2 steps. 

\vspace{.3cm}
(1) First by assuming the formula for $\dvev{2m-1}$ in \eqref{t2m-1}, we shall establish the formula \eqref{t2m} 
for $\dvev{2m}$, via the identity $[2m] \dvev{2m} =\mathfrak{t} \cdot \dvev{2m-1}$ in \eqref{eq:tt}. 

Recall the formula \eqref{t2m-1} for $\dvev{2m-1}$.
Using $\mathfrak{t}=\Y+F$ and applying \eqref{FYn} to $F\Y^{(a)}$ we have
\begin{align}  \label{tt2m-1}
\mathfrak{t}\cdot \dvev{2m-1} 
&= \sum_{c=0}^{m-1} \sum_{a=0}^{2m-1-2c} q^{\binom{2c+1}{2} -a(2m-1-2c-a)}  \mathfrak{t} \Y^{(a)}  \qbinom{\kk; 1-m}{c}  F^{(2m-1-2c-a)}
\\
&= \sum_{c=0}^{m-1} \sum_{a=0}^{2m-1-2c} q^{\binom{2c+1}{2} -a(2m-1-2c-a)} \cdot 
\notag \\
&\qquad \quad \left( \Y \Y^{(a)}  +q^{-2a} \Y^{(a)} F  +\Y^{(a-1)} \frac{q^{3-3a} K^{-2} - q^{1-a} }{q^2-1} \right) \qbinom{\kk; 1-m}{c}  F^{(2m-1-2c-a)}  
 \notag \\
&= \sum_{c=0}^{m-1} \sum_{a=0}^{2m-1-2c} q^{\binom{2c+1}{2} -a(2m-1-2c-a)} \cdot 
\notag \\
& \left( [a+1] \Y^{(a+1)}  \qbinom{\kk; 1-m}{c}  F^{(2m-1-2c-a)}
 +q^{-2a} [2m-2c-a] \Y^{(a)}  \qbinom{\kk; -m}{c}  F^{(2m-2c-a)} \right.
 \notag \\
& \qquad +
\left. \Y^{(a-1)} \frac{q^{3-3a} K^{-2} - q^{1-a} }{q^2-1} \qbinom{\kk; 1-m}{c}  F^{(2m-1-2c-a)} \right). 
\notag
\end{align}

We  reorganize the formula \eqref{tt2m-1} in the following form
\[
[2m] \cdot \dvev{2m}=
\mathfrak{t} \cdot \dvev{2m-1} =  \sum_{c=0}^m \sum_{a=0}^{2m-2c} \Y^{(a)} f_{a,c}(\kk) F^{(2m-2c-a)}, 
\]
where
\begin{align*}
f_{a,c}(\kk) &= q^{\binom{2c+1}{2} -(a-1)(2m-2c-a)}   [a]\,  \qbinom{\kk; 1-m}{c} \\
&\quad + \left( q^{\binom{2c+1}{2} -a(2m-1-2c-a) -2a} [2m-2c-a]\, \qbinom{\kk; -m}{c} \right. \\
&\qquad\quad \left. +q^{\binom{2c-1}{2} -(a+1)(2m-2c-a)}  \frac{q^{-3a} K^{-2} - q^{-a} }{q^2-1} \qbinom{\kk; 1-m}{c-1} \right).
\end{align*}
A direct computation gives us
\begin{align*}
f_{a,c}(\kk) &=  q^{\binom{2c}{2} - a(2m-2c-a)} q^{2m-a}  [a] \,  \qbinom{\kk; 1-m}{c} +q^{\binom{2c}{2}- a(2m-2c-a)} \cdot  \\
& \qquad \cdot \left( q^{2c-a}[2m-2c-a] \frac{q^{-4m} K^{-2}-1}{q^{4c}-1}  +q^{1+a-2m} \frac{q^{-3a} K^{-2}-q^{-a}}{q^{2}-1}  \right)  \qbinom{\kk; 1-m}{c-1}
\\
&=  q^{\binom{2c}{2} - a(2m-2c-a)} q^{2m-a}  [a] \,  \qbinom{\kk; 1-m}{c}
+q^{\binom{2c}{2} - a(2m-2c-a)} q^{-a} [2m-a] \,  \qbinom{\kk; 1-m}{c}
\\
&=  q^{\binom{2c}{2} - a(2m-2c-a)}  [2m]\,  \qbinom{\kk; 1-m}{c}.
\end{align*}
Hence we have obtained the formula \eqref{t2m}  for $\dvev{2m}$.

\vspace{.3cm}
(2) 
Now by assuming the formula for $\dvev{2m}$ in \eqref{t2m}, we shall establish the following formula (with $m$ in \eqref{t2m-1} replaced by $m+1$)
\begin{align}
\dvev{2m+1} &=  \sum_{c=0}^m \sum_{a=0}^{2m+1-2c} q^{\binom{2c+1}{2} -a(2m+1-2c-a)} \Y^{(a)}  \qbinom{\kk;-m}{c}  F^{(2m+1-2c-a)}. 
\label{t2m+1}
\end{align}

Recall the formula for $\dvev{2m}$ in \eqref{t2m}.
Using  $\mathfrak{t}=\Y+F$ and applying \eqref{FYn} to $F\Y^{(a)}$ we have
\begin{align*}
\mathfrak{t}\cdot \dvev{2m} &=   \sum_{c=0}^m \sum_{a=0}^{2m-2c} q^{\binom{2c}{2} -a(2m-2c-a)} \mathfrak{t} \Y^{(a)}  \qbinom{\kk; 1-m}{c}  F^{(2m-2c-a)} \\
&=  \sum_{c=0}^m \sum_{a=0}^{2m-2c} q^{\binom{2c}{2} -a(2m-2c-a)}  \cdot \notag
 \\
&\qquad \cdot \left( \Y \Y^{(a)}   +q^{-2a} \Y^{(a)} F   +\Y^{(a-1)} \frac{q^{3-3a} K^{-2} - q^{1-a} }{q^2-1} \right) \qbinom{\kk; 1-m}{c}  F^{(2m-2c-a)}.
\end{align*}
We rewrite this as
\begin{align}  \label{tt2m}
\mathfrak{t}\cdot \dvev{2m} 
 &=  \sum_{c=0}^m \sum_{a=0}^{2m-2c} q^{\binom{2c}{2} -a(2m-2c-a)} \cdot 
  \left ( [a+1] \Y^{(a+1)}  \qbinom{\kk; 1-m}{c}  F^{(2m-2c-a)} \right.  \\
 &\qquad\qquad\qquad 
 +q^{-2a} [2m+1-2c-a] \Y^{(a)}  \qbinom{\kk; -m}{c}  F^{(2m+1-2c-a)}  
 \notag \\
& \qquad\qquad\qquad  \left. +\Y^{(a-1)} \frac{q^{3-3a} K^{-2} - q^{1-a} }{q^2-1} \qbinom{\kk; 1-m}{c}  F^{(2m-2c-a)} \right).
\notag
\end{align}
We shall use \eqref{eq:tt}, \eqref{tt2m} and \eqref{t2m-1} to obtain a formula of the form
\begin{equation} \label{eq:ttsum}
[2m+1] \dvev{2m+1} =\mathfrak{t} \cdot \dvev{2m} -[2m] \dvev{2m-1}
  =   
  \sum_{c=0}^m \sum_{a=0}^{2m+1-2c} \Y^{(a)} g_{a,c}(\kk) F^{(2m+1-2c-a)}, 
\end{equation}
for some suitable $g_{a,c}(\kk)$. Then we have
\begin{align*}
g_{a,c}(\kk) &= q^{\binom{2c}{2} -(a-1)(2m+1-2c-a)}   [a]\,  \qbinom{\kk; 1-m}{c} \\
&\quad +  q^{\binom{2c}{2} -a(2m-2c-a) -2a} [2m+1-2c-a]\, \qbinom{\kk; -m}{c}   \\
&\quad   +q^{\binom{2c-2}{2} -(a+1)(2m+1-2c-a)}  \frac{q^{-3a} K^{-2} - q^{-a} }{q^2-1} \qbinom{\kk; 1-m}{c-1} 
\\
&\quad   -q^{\binom{2c-1}{2} - a(2m+1-2c-a)} [2m] \qbinom{\kk; 1-m}{c-1} 
\\
&= q^{\binom{2c+1}{2} - a(2m+1-2c-a)}  q^{-2c-a} [2m+1-2c-a]\, \qbinom{\kk; -m}{c}   + q^{\binom{2c+1}{2} - a(2m+1-2c-a)}  X, 
\end{align*}
where  
\begin{align*}
X&=q^{2m+1-4c-a}   [a]\,  \qbinom{\kk; 1-m}{c}  
\\
&\qquad +q^{-2m+a-4c+2} \frac{q^{-3a} K^{-2} - q^{-a} }{q^2-1} \qbinom{\kk; 1-m}{c-1} 
-q^{1-4c} [2m] \qbinom{\kk; 1-m}{c-1}. 
\end{align*}
A direct computation allows us to simplify the expression for $X$ as follows:
\begin{align*}
X&= \left( q^{2m+1-4c-a}   [a]  \frac{q^{4c-4m} K^{-2} - 1}{q^{4c}-1}
 +q^{-2m+a-4c+2} \frac{q^{-3a} K^{-2} - q^{-a} }{q^2-1}  -q^{1-4c} [2m]
\right) \qbinom{\kk; 1-m}{c-1}
\\
&= q^{2m-2c-a+1}  [2c+a]   \frac{q^{-4m} K^{-2} - 1}{q^{4c}-1}   \qbinom{\kk; 1-m}{c-1}  %corrected: q^2 >>> q^{4c}
\\
&= q^{2m-2c-a+1}  [2c+a]    \qbinom{\kk; -m}{c}. 
\end{align*}

Hence, we  obtain
\begin{align*}
g_{a,c}(\kk) &= q^{\binom{2c+1}{2} - a(2m+1-2c-a)}  q^{-2c-a} [2m+1-2c-a]\, \qbinom{\kk; -m}{c} 
\\
&\qquad + q^{\binom{2c+1}{2} - a(2m+1-2c-a)}  q^{2m-2c-a+1}  [2c+a]    \qbinom{\kk; -m}{c}
\\ &=  q^{\binom{2c+1}{2} - a(2m+1-2c-a)}  [2m+1]\,  \qbinom{\kk; -m}{c}.
\end{align*}
Recalling the identity \eqref{eq:ttsum}, we have proved the formula \eqref{t2m+1}  for $\dvev{2m+1}$, and
hence completed the proof of Theorem~\ref{thm:iDP}.

\section{Formulae for $\imath$-divided powers $\dv{n}$ and $\imath$-canonical basis on $L(2\la+1)$}
  \label{sec:dv:odd}

In this section we formulate the $\imath$-divided powers $\dv{n}$ (of odd weights), and established closed formulae for it in $\U$.

\subsection{The $\imath$-divided powers $\dv{n}$}

We consider an odd weight variant of the $\imath$-divided powers of $\mathfrak{t}$,  denoted by $\dv{n}$.
Recall $\mathfrak{t} =\Y+F$. The formulae are given by  
\begin{align}
  \label{eq:dv}
  \begin{split}
%T^{\rm odd}_{2a} &= {t(t - [-2a+2])(t-[-2a+4]) \cdots (t-[2a-4]) (t - [2a-2]) \over [2a]!},\\
%T^{\rm odd}_{2a+1} & = {  (t - [-2a])(t-[-2a+2]) \cdots (t-[2a-2]) (t - [2a])  \over [2a+1]!},\\
\dv{2a} &= { (\mathfrak{t} - [-2a+1] ) (\mathfrak{t} - [-2a+3]) \cdots (\mathfrak{t} - [2a-3]) (\mathfrak{t}-[2a-1]) \over [2a]!},\\
\dv{2a+1} &= { \mathfrak{t}(\mathfrak{t} - [-2a+1] ) (\mathfrak{t} - [-2a+3]) \cdots (\mathfrak{t} - [2a-3]) (\mathfrak{t}-[2a-1]) \over [2a+1]!},
\quad \text{ for } a\ge 0.
  \end{split}
\end{align}
These formulae formally coincide with the formulae for
$T^{\rm ev}_{n}$ in \cite[Conjecture~4.13]{BW13}, but $\mathfrak{t}$ here is different from $t$ therein. 

The $\imath$-divided powers  
$\dv{n}$ satisfy (and are determined by) the following recursive relations: 
\begin{align}  %1b
\label{eq:ttodd2}
  \begin{split}
\mathfrak{t} \cdot \dv{2a} &=  [2a+1] \dv{2a+1} ,
\\
\mathfrak{t} \cdot \dv{2a+1} &=[2a+2] \dv{2a+2}+[2a+1] \dv{2a}, \quad \text{ for } a \ge 0.
  \end{split}
\end{align}
Note $\dv{0}=1, \dv{1}=\mathfrak{t}, \dv{2} =(\mathfrak{t}^2-1)/[2]$ and $\dv{3} =\mathfrak{t}(\mathfrak{t}^2-1)/[3]!$. 

\subsection{The algebra $\DA$}

Set, for $n\ge 1, a\in \Z$, 
\begin{equation}   \label{brace}
\LR{\kk;a}{0}=1, 
\qquad
\LR{\kk;a}{n}= \prod_{i=1}^n \frac{q^{4a+4i-4} K^{-2}-q^2}{q^{4i}-1}, 
\qquad \llbracket \kk;a\rrbracket = \LR{\kk;a}{1}.
\end{equation}
%\red{Redefine $\LR{\kk;a}{n}' := \prod_{i=1}^n \frac{q^{4a+4i-4} K^{-2}-q^{-2}}{q^{4i}-1} =$ which equals $q^{-4n} \LR{\kk;a+1}{n}$?}
Note $\kk=q[2] \llbracket \kk;0 \rrbracket+1$. We have
\begin{equation}
 \label{kqbinom2}
\LR{\kk;a}{n} \one_{2\la-1} = q^{2n(a-\la)} \qbinom{a-\la-1+n}{n}_{q^2} \one_{2\la-1} \in \UAdot_\odd. 
 \end{equation}
It follows from \eqref{sl2} and \eqref{brace} that, for $n\ge 0$ and $a\in \Z$, 
\begin{align}
\LR{\kk;a}{n} F = F  \LR{\kk;a+1}{n},
\qquad
\LR{\kk;a}{n} \Y =\Y \LR{\kk;a-1}{n}.
\end{align}
Recall the $\Q(q)$-subalgebra $\B$ of $\U$.  
Denote by $\DA$ the $\A$-subalgebra of $\B$ with $1$ generated by $\Y^{(n)}, F^{(n)}, K^{-1}$ and $\LR{\kk;a}{n}$, 
for all $n \ge 1$ and $a\in \Z$. 

\subsection{The $\Y \kk F$-formula for $\dv{n}$}

\begin{thm}  \label{thm:iDP:odd}
For $m\ge 0$, we have
\begin{align}
\dv{2m} &= \sum_{c=0}^m \sum_{a=0}^{2m-2c} q^{\binom{2c}{2} -a(2m-2c-a)}
\Y^{(a)}  \LR{\kk;1-m}{c}  F^{(2m-2c-a)}, 
\label{todd2m}  
\\
\dv{2m+1} &=  \sum_{c=0}^{m} \sum_{a=0}^{2m+1-2c} q^{\binom{2c-1}{2}-1 -a(2m+1-2c-a)}
\Y^{(a)}  \LR{\kk; 1-m}{c}  F^{(2m+1-2c-a)}. % -m >>> 1-m
\label{todd2m+1}   
\end{align}
\end{thm}
The proof of this theorem will be given in \S\ref{sec:proof:iDP:odd} below. 
Note  that the summation over $a$ for $c=0$ in \eqref{todd2m}--\eqref{todd2m+1} is equal to $b^{(2m)}$ and $b^{(2m+1)}$, respectively. 
It follows from Theorem~\ref{thm:iDP:odd} that $\dv{n} \in \DA$ for all $n$. 

\begin{example}  
We have the following examples of $\dv{n}$, for $2\le n\le 5$:
\begin{align*}
\dv{2}  
&= b^{(2)} + q\llbracket \kk;0 \rrbracket,
\\
\dv{3}  
&= b^{(3)} + q^{-1} \llbracket \kk;0 \rrbracket F + q^{-1} \Y \llbracket \kk;0 \rrbracket,  % 
\\
\dv{4}  
&= b^{(4)} +  q \Y^{(2)} \llbracket \kk;-1 \rrbracket + q \llbracket \kk;-1 \rrbracket F^{(2)} + \Y \llbracket \kk;-1 \rrbracket F + q^6 \LR{\kk;-1}{2},
\\
\dv{5}  
&= b^{(5)} + q^{-1} \Y^{(3)} \llbracket \kk;-1 \rrbracket + q^{-1} \llbracket \kk;-1 \rrbracket F^{(3)} + q^{-3} \Y^{(2)} \llbracket \kk;-1 \rrbracket F + q^{-3} \Y \llbracket \kk;-1 \rrbracket F^{(2)} 
\\ &\qquad + q^2 \Y \LR{\kk;-1}{2} + q^2 \LR{\kk;-1}{2} F. % 
\end{align*}
\end{example}

\subsection{Some reformulations}

\begin{lem}
  \label{lem:anti2}
The anti-involution $\vs$ on  the $\Q$-algebra $\B$ sends 
\[
\dv{n} \mapsto \dv{n},\quad 
\LR{\kk;a}{n}\mapsto (-1)^n q^{2n(n-1)}  \LR{\kk;2-a-n}{n}, \quad \forall a \in \Z,\; n\in \N. 
\]
\end{lem}

\begin{proof}
The second formula follows from a direct computation:
\begin{align*}
\vs \left( \LR{\kk;a}{n} \right)
&= \prod_{i=1}^n \frac{q^{-4a-4i+4} K^{-2}-q^{-2}}{q^{-4i}-1}
\\
&= q^{-4n} \prod_{i=1}^n (- q^{4i})  \cdot \prod_{i=1}^n \frac{q^{-4a-4i+8} K^{-2}-q^{2}}{q^{4i}-1}
= (-1)^n q^{2n(n-1)}  \LR{\kk;2-a-n}{n}.
\end{align*}
We know by  Lemma~\ref{lem:anti} that $\dv{1} =t=\Y+F$ is fixed by $\vs$.
The identity $\vs(\dv{n}) =\dv{n}$ follows now from  the recursive relations \eqref{eq:ttodd2}.
\end{proof}

We have the following $F\kk\Y$-expansion formulae, which is easily obtained 
from the formulae in Theorem~\ref{thm:iDP:odd} by applying the anti-involution $\vs$ in Lemma~\ref{lem:anti} and Lemma~\ref{lem:anti2}. The proof is skipped. 

\begin{prop}  \label{iDP:odd:FkY}
For $m\ge 0$, we have
\begin{align*}
\dv{2m} &= \sum_{c=0}^m \sum_{a=0}^{2m-2c} (-1)^c q^{-c+ a(2m-2c-a)}
F^{(a)}  \LR{\kk;1+m-c}{c}  \Y^{(2m-2c-a)}, 
\\
\dv{2m+1} &=  \sum_{c=0}^{m} \sum_{a=0}^{2m+1-2c} (-1)^c q^{c+ a(2m+1-2c-a)}
F^{(a)}  \LR{\kk;1+m-c}{c}  \Y^{(2m+1-2c-a)}. 
\end{align*}
\end{prop}

Below is a modified version of Theorem~\ref{thm:iDP:odd}.

\begin{prop}  \label{iDP:odd:dot}
For $m\ge 0$ and $\la \in \Z$, we have
\begin{align*}
\dv{2m}  \one_{2\la-1}&= \sum_{c=0}^m \sum_{a=0}^{2m-2c} q^{ 2(a+c)(m-a-\la) -2ac+a -\binom{2c}{2} }  
\qbinom{m-c-a-\la}{c}_{q^2} E^{(a)}  F^{(2m-2c-a)} \one_{2\la-1}, 
\\
\dv{2m+1}  \one_{2\la-1}&=  \sum_{c=0}^{m} \sum_{a=0}^{2m+1-2c} 
\\
& \qquad q^{2(a+c)(m-a-\la)-2ac+2a  -\binom{2c}{2}} \qbinom{m-c-a-\la+1}{c}_{q^2} E^{(a)}  F^{(2m+1-2c-a)}  \one_{2\la-1}. 
\end{align*}
In particular, we have $\dv{n}\one_{2\la-1} \in \UAdot_\odd$, for all $n\in \N$. 
\end{prop}

%\begin{align*}
%\dv{2m}  \one_{2\la-1}&= \sum_{c=0}^m \sum_{a=0}^{2m-2c} q^{-\binom{2c}{2}-a(2m-a) + 2c(m-\la)}
%\qbinom{m-c-a-\la}{c}_{q^2} \Y^{(a)}  F^{(2m-2c-a)} \one_{2\la-1}, 
%\\
%\dv{2m+1}  \one_{2\la-1}&=  \sum_{c=0}^{m} \sum_{a=0}^{2m+1-2c} q^{-\binom{2c}{2}-a(2m-a+1) + 2c(m-\la)}
%\qbinom{m-c-a-\la}{c}_{q^2} \Y^{(a)}  F^{(2m+1-2c-a)}  \one_{2\la-1}. 
%\end{align*}

\begin{proof}
Note $F^{(2m-2c-a)} \one_{2\la-1} = \one_{2(\la-2m+2c+a)-1} F^{(2m-2c-a)}$, cf. \eqref{todd2m}.
By \eqref{kqbinom2} and Lemma~\ref{dpY}, we have
\begin{align*}
& q^{\binom{2c}{2} -a(2m-2c-a)} \Y^{(a)} \LR{\kk;1-m}{c}  \one_{2(\la-2m+2c+a)-1}
\\
&=q^{\binom{2c}{2} -a(2m-2c-a)} 
q^{a -2a(\la -2m+2c+a) -a^2}
q^{2c(m-2c-a-\la+1)} \qbinom{m-c-a-\la}{c}_{q^2} E^{(a)} \one_{2(\la-2m+2c+a)-1}
\\
&=q^{ 2(a+c)(m-a-\la) -2ac+a -\binom{2c}{2} }  
 \qbinom{m-c-a-\la}{c}_{q^2} E^{(a)} \one_{2(\la-2m+2c+a)-1}.
\end{align*}
This proves the first formula. 

On the other hand, 
we have $F^{(2m+1-2c-a)} \one_{2\la-1} = \one_{2(\la-2m+2c+a-1)-1} F^{(2m+1-2c-a)}$, cf. \eqref{todd2m+1}.
It follows from \eqref{kqbinom2} and Lemma~\ref{dpY} that
\begin{align*}
& q^{\binom{2c-1}{2}-1 -a(2m+1-2c-a)} \Y^{(a)} 
\LR{\kk; 1-m}{c}  \one_{2(\la-2m+2c+a-1)-1}  %%corrected  -m >>>1-m
\\
&= q^{\binom{2c-1}{2}-1 -a(2m+1-2c-a)} 
q^{a +2a(2m-\la-2c-a+1) -a^2} \cdot
\\
& \qquad\qquad q^{2c(m-2c-a-\la+2)} \qbinom{m-c-a-\la+1}{c}_{q^2} E^{(a)} \one_{2(\la-2m+2c+a-1)-1} 
    %%corrected \one_{2(\la-2m+2c+a)-1}  >>> \one_{2(\la-2m+2c+a-1)-1} 
\\
&= q^{2(a+c)(m-a-\la)-2ac+2a  -\binom{2c}{2}} \qbinom{m-c-a-\la+1}{c}_{q^2} E^{(a)} \one_{2(\la-2m+2c+a-1)-1}.
    %%corrected \one_{2(\la-2m+2c+a)-1}  >>> \one_{2(\la-2m+2c+a-1)-1} 
\end{align*}
The second formula follows. 
\end{proof}

\subsection{The $\imath$-canonical basis for simple $\U$-modules $L(2\la+1)$}
  \label{sec:iCB:odd}
  
 % \red{Define  $\dot{\U}^\imath_{\odd}$. }
  
  Recall $\cbinom{m}{c}$ from \eqref{cbinom}.  
We shall establish the following conjecture by Bao and the second author.
 \begin{thm}  \label{thm:iCB:odd}
 $\quad$
\begin{enumerate}
\item
The set $\{\dv{n}\, \vodd \mid 0\le n \le 2\la+1 \}$ forms the $\imath$-canonical basis for $L(2\la +1)$, for each $\la \in \N$. 
Moreover, $\dv{n}\, \vodd =0$, for $n>2\la+1$.

\item
The set $\{\dv{n} \mid n \in \N \}$ forms the $\imath$-canonical basis for $\Ui$. 
\end{enumerate}
\end{thm}
 
 \begin{proof}
Let $\la, m \in \N$.  It follows by a direct computation using Proposition~\ref{iDP:odd:FkY} that 
\begin{align} 
\dv{2m} \vodd 
%&= \sum_{c=0}^m (-1)^c q^{-c} F^{(2m-2c)}  \qbinom{\kk;1+m-c}{c} \vodd \notag \\
& = \sum_{c=0}^m q^{-2c^2-c} \cbinom{m-\la-c}{c} F^{(2m-2c)}  \vodd;
  \label{dv:2m-mod} \\
\dv{2m+1} \vodd 
%&= \sum_{c=0}^{m-1} (-1)^c q^{c} F^{(2m+1-2c)}  \qbinom{\kk;1+m-c}{c} \vodd \notag \\
& = \sum_{c=0}^m q^{-2c^2+c} \cbinom{m-\la-c}{c} F^{(2m+1-2c)}  \vodd.
  \label{dv:2m+1-mod}
\end{align}
We observe that 
\begin{align}  \label{odd=0}
\begin{split}
  F^{(2m-2c)}  \vodd = F^{(2m+1-2c)}  \vodd =0, & \quad \text{ if } c<m-\la.
  \end{split}
\end{align}
It follows from \eqref{cbin=0}, \eqref{dv:2m-mod}, \eqref{dv:2m+1-mod} and \eqref{odd=0}  that $\dv{2m} \vodd =\dv{2m+1} \vodd =0$, for $m\ge \la+1$;
moreover, for $m\le \la$, we have
 \begin{align*}
 \dv{2m} \vodd  & \in F^{(2m)}  \vodd + \sum_{c\ge 1}q^{-1} \N[q^{-1}] F^{(2m-2c)} \vodd,
 \\
 \dv{2m+1} \vodd  & \in F^{(2m+1)}  \vodd + \sum_{c\ge 1}q^{-1} \N[q^{-1}] F^{(2m+1-2c)} \vodd.
 \end{align*}
 
 Therefore, the first statement follows by the characterization of the $\imath$-canonical basis for $L(2\la+1)$.
 The second statement follows now from the definition of the $\imath$-canonical basis for $\Ui$ using
the projective system $\{ L(2\la+1) \}_{\la\in \N}$; cf. \cite[\S6]{BW16} (also cf. \cite[\S4]{BW13}). 
 \end{proof}

%
%
%\iffalse
%%%%%%
\subsection{Proof of Theorem~\ref{thm:iDP:odd} }
\label{sec:proof:iDP:odd}

We prove the formulae for $\dv{n}$ by induction on $n$. The base case for $n=1$ is clear.
The induction is carried out in 2 steps. 

\vspace{.3cm}
(1) First by assuming the formula for $\dv{2m}$ in \eqref{todd2m}, we shall establish the formula \eqref{todd2m+1} 
for $\dv{2m+1}$, via the identity $[2m+1] \dv{2m+1} = \mathfrak{t} \cdot \dv{2m}$ in \eqref{eq:ttodd2}. 

Recall the formula \eqref{todd2m} for $\dv{2m}$.
Using $\mathfrak{t}=\Y+F$ and applying \eqref{FYn} to $F\Y^{(a)}$ we have
\begin{align}  \label{ttodd2m}
\mathfrak{t}\cdot \dv{2m} 
&= \sum_{c=0}^{m} \sum_{a=0}^{2m-2c} q^{\binom{2c}{2} -a(2m-2c-a)}  \mathfrak{t} \Y^{(a)}  \LR{\kk;1-m}{c} F^{(2m-2c-a)}
\\
&= \sum_{c=0}^{m} \sum_{a=0}^{2m-2c} q^{\binom{2c}{2} -a(2m-2c-a)} \cdot 
\notag \\
&\qquad \quad \left( \Y \Y^{(a)}  +q^{-2a} \Y^{(a)} F  +\Y^{(a-1)} \frac{q^{3-3a} K^{-2} - q^{1-a} }{q^2-1} \right) \LR{\kk;1-m}{c} F^{(2m-2c-a)}  
 \notag \\
&= \sum_{c=0}^{m} \sum_{a=0}^{2m-2c} q^{\binom{2c}{2} -a(2m-2c-a)} \cdot 
\notag \\
& \left( [a+1] \Y^{(a+1)}  \LR{\kk;1-m}{c} F^{(2m-2c-a)}
 +q^{-2a} [2m+1-2c-a] \Y^{(a)}  \LR{\kk;-m}{c} F^{(2m+1-2c-a)} \right.
 \notag \\
& \qquad +
\left. \Y^{(a-1)} \frac{q^{3-3a} K^{-2} - q^{1-a} }{q^2-1} \LR{\kk;1-m}{c} F^{(2m-2c-a)} \right). 
\notag
\end{align}

We  reorganize the formula \eqref{ttodd2m} in the following form
\[
[2m+1] \dv{2m+1} = \mathfrak{t} \cdot \dv{2m} =  \sum_{c=0}^m \sum_{a=0}^{2m+1-2c} \Y^{(a)} \texttt{f}_{a,c}(\kk) F^{(2m+1-2c-a)}, 
\]
where
\begin{align*}
\texttt{f}_{a,c}(\kk) &= q^{\binom{2c}{2} -(a-1)(2m+1-2c-a)}   [a]\,  \LR{\kk;1-m}{c}\\
&\quad + \left( q^{\binom{2c}{2} -a(2m-2c-a) -2a} [2m+1-2c-a]\, \LR{\kk;-m}{c}\right. \\
&\qquad\quad \left. +q^{\binom{2c-2}{2} -(a+1)(2m+1-2c-a)}  \frac{q^{-3a} K^{-2} - q^{-a} }{q^2-1} \LR{\kk;1-m}{c-1}\right).
\end{align*}
A direct computation gives us
\begin{align*}
\texttt{f}_{a,c}(\kk) &=  q^{\binom{2c-1}{2} -1 - a(2m+1-2c-a)} q^{2m+1-a}  [a] \,  \LR{\kk;1-m}{c}+q^{\binom{2c-1}{2} -1 - a(2m+1-2c-a)} \cdot  \\
&\quad\qquad \cdot \left( q^{2c-a}[2m+1-2c-a] \frac{q^{-4m} K^{-2}-q^2}{q^{4c}-1}  +q^{2+a-2m} \frac{q^{-3a} K^{-2}-q^{-a}}{q^{2}-1}  \right)  \LR{\kk;1-m}{c-1}
\\
&=  q^{\binom{2c-1}{2} -1 - a(2m+1-2c-a)} q^{2m+1-a}  [a] \,  \LR{\kk;1-m}{c}
\\
&\qquad +q^{\binom{2c-1}{2} -1 - a(2m+1-2c-a)} q^{-a} [2m+1-a] \,  \LR{\kk;1-m}{c}
\\
&=  q^{\binom{2c-1}{2} -1 - a(2m+1-2c-a)}  [2m+1]\,  \LR{\kk;1-m}{c}.
\end{align*}
Hence we have obtained the formula \eqref{todd2m+1}  for $\dv{2m+1}$.

\vspace{.3cm}
(2) 
Now by assuming the formula for $\dv{2m+1}$ in \eqref{todd2m+1}, we shall establish the following formula (with $m$ in \eqref{todd2m} replaced by $m+1$)
\begin{align}
\dv{2m+2} &=  \sum_{c=0}^{m+1} \sum_{a=0}^{2m+2-2c} q^{\binom{2c}{2} -a(2m+2-2c-a)} \Y^{(a)}  \LR{\kk;-m}{c} F^{(2m+2-2c-a)}. 
\label{todd2m+2}
\end{align}

Recall the formula \eqref{todd2m+1} for $\dv{2m+1}$.
Using  $\mathfrak{t}=\Y+F$ and applying \eqref{FYn} to $F\Y^{(a)}$ we have
\begin{align*}
\mathfrak{t} \cdot \dv{2m+1} &=   \sum_{c=0}^m \sum_{a=0}^{2m+1-2c} q^{\binom{2c-1}{2}-1 -a(2m+1-2c-a)}   \mathfrak{t} \Y^{(a)}  \LR{\kk;1-m}{c} F^{(2m+1-2c-a)} \\
&=  \sum_{c=0}^m \sum_{a=0}^{2m+1-2c} q^{\binom{2c-1}{2}-1 -a(2m+1-2c-a)}  \cdot \notag
 \\
&\qquad \cdot \left( \Y \Y^{(a)}   +q^{-2a} \Y^{(a)} F   +\Y^{(a-1)} \frac{q^{3-3a} K^{-2} - q^{1-a} }{q^2-1} \right) \LR{\kk;1-m}{c} F^{(2m+1-2c-a)}.
\end{align*}
We rewrite this as
\begin{align}  \label{ttodd2m+1}
\mathfrak{t} \cdot \dv{2m+1}
 &=  \sum_{c=0}^m \sum_{a=0}^{2m+1-2c} q^{\binom{2c-1}{2}-1 -a(2m+1-2c-a)} \cdot 
  \left ( [a+1] \Y^{(a+1)} \LR{\kk;1-m}{c} F^{(2m+1-2c-a)} \right.  \\
 &\qquad\qquad\qquad 
 +q^{-2a} [2m+2-2c-a] \Y^{(a)}  \LR{\kk;-m}{c} F^{(2m+2-2c-a)}  
 \notag \\
& \qquad\qquad\qquad  \left. +\Y^{(a-1)} \frac{q^{3-3a} K^{-2} - q^{1-a} }{q^2-1} \LR{\kk;1-m}{c} F^{(2m+1-2c-a)} \right).
\notag
\end{align}
We shall use \eqref{eq:ttodd2}, \eqref{ttodd2m+1} and \eqref{todd2m} to obtain a formula of the form
\begin{equation} \label{eq:ttoddsum}
[2m+2] \dv{2m+1} = \mathfrak{t}\cdot \dv{2m+1} -[2m+1] \dv{2m}
  =   
  \sum_{c=0}^{m+1} \sum_{a=0}^{2m+2-2c} \Y^{(a)} \texttt{g}_{a,c}(\kk) F^{(2m+2-2c-a)}, 
\end{equation}
for some suitable $\texttt{g}_{a,c}(\kk)$. Then we have
\begin{align*}
\texttt{g}_{a,c}(\kk) &= q^{\binom{2c-1}{2}-1 -(a-1)(2m+2-2c-a)}   [a]\,  \LR{\kk;1-m}{c}\\
&\quad +  q^{\binom{2c-1}{2}-1 -a(2m+1-2c-a) -2a} [2m+2-2c-a]\, \LR{\kk;-m}{c} \\
&\quad   +q^{\binom{2c-3}{2}-1 -(a+1)(2m+2-2c-a)}  \frac{q^{-3a} K^{-2} - q^{-a} }{q^2-1} \LR{\kk;1-m}{c-1}
\\
&\quad   -q^{\binom{2c-2}{2} - a(2m+2-2c-a)} [2m+1] \LR{\kk;1-m}{c-1}
\\
&= q^{\binom{2c}{2} - a(2m+2-2c-a)}  q^{-2c-a} [2m+2-2c-a]\, \LR{\kk;-m}{c}  + q^{\binom{2c}{2} - a(2m+2-2c-a)}  \texttt{X}, 
\end{align*}
where  
\begin{align*}
\texttt{X}&=q^{2m+2-4c-a}   [a]\,  \LR{\kk;1-m}{c}
\\
&\qquad +q^{-2m+3-4c+a} \frac{q^{-3a} K^{-2} - q^{-a} }{q^2-1} \LR{\kk;1-m}{c-1}
-q^{3-4c} [2m+1] \LR{\kk;1-m}{c-1}.
\end{align*}
A direct computation allows us to simplify the expression for $\texttt{X}$ as follows:
\begin{align*}
\texttt{X}&= \left( q^{2m+2-4c-a}   [a]  \frac{q^{4c-4m} K^{-2} - q^2}{q^{4c}-1}
 +q^{-2m+3-4c+a} \frac{q^{-3a} K^{-2} - q^{-a} }{q^2-1}  -q^{3-4c} [2m+1]
\right) \LR{\kk;1-m}{c-1}
\\
&= q^{2m+2-2c-a}  [2c+a]   \frac{q^{-4m} K^{-2} - q^2}{q^{4c}-1}   \LR{\kk;1-m}{c-1}
\\
&= q^{2m+2-2c-a}  [2c+a]    \LR{\kk;-m}{c}. 
\end{align*}

Hence, we  obtain
\begin{align*}
\texttt{g}_{a,c}(\kk) &= q^{\binom{2c}{2} - a(2m+2-2c-a)}  q^{-2c-a} [2m+2-2c-a]\, \LR{\kk;-m}{c}
\\
&\qquad + q^{\binom{2c}{2} - a(2m+2-2c-a)}  q^{2m+2-2c-a}  [2c+a]    \LR{\kk;-m}{c}
\\ &=  q^{\binom{2c}{2} - a(2m+2-2c-a)}  [2m+2]\,  \LR{\kk;-m}{c}.
\end{align*}
Recalling the identity \eqref{eq:ttoddsum}, we have proved the formula \eqref{todd2m+2}  for $\dv{2m+2}$, and
hence completed the proof of Theorem~\ref{thm:iDP:odd}. 
%%%%%%\fi
%

%%
%%
\section{Formulae for $\imath$-divided powers $\dvp{n}$ and $\imath$-canonical basis on $L(2\la)$}
  \label{sec:dvK:even}

In this section we formulate a variant of $\imath$-divided powers (of even weights) starting with $t$ in \eqref{def:t} below, and established closed formulae for it in $\U$.

\subsection{The $\imath$-divided powers $\dvp{n}$}

Set 
\begin{equation}  \label{def:t}
t =\Y+F+K^{-1}.
\end{equation}
That is, $t= \mathfrak{t} +K^{-1}$. Set $\dvp{0}=1, \dvp{1}=t$. 
The $\imath$-divided powers $\dvp{n}$ (of even weights) are determined by the  recursive relations \eqref{eq:ttodd2:dvp}. 
Note $\dvp{2} =(t^2-1)/[2]$ and $\dvp{3} =t(t^2-1)/[3]!$. 
Note that $\dvp{n}$ satisfy the same recursive relations \eqref{eq:ttodd2} for $\dv{n}$ and so are given by the same formula \eqref{eq:dv} (with $\mathfrak{t}$ replaced by $t$).

\subsection{The $\Y \kk F$-formula for $\dvp{n}$}

Recall $\qbinom{\kk;a}{n}$ from \eqref{kbinom}. 

\begin{thm}
  \label{thm:iDPK}
For $m\ge 0$, we have 
\begin{align}
\dvp{2m} &= \sum_{c=0}^m \sum_{a=0}^{2m-2c} q^{\binom{2c+1}{2} -a(2m-2c-a)}
\Y^{(a)}  \qbinom{\kk;1-m}{c}  F^{(2m-2c-a)}
\label{dvp2m}
\\ 
&\qquad 
+ \sum_{c=0}^{m-1} \sum_{a=0}^{2m-1-2c} q^{\binom{2c+2}{2}-2m -a(2m-1-2c-a)} \Y^{(a)}  \qbinom{\kk;1-m}{c}K^{-1}  F^{(2m-1-2c-a)}, 
\notag \\
\dvp{2m+1} &=  \sum_{c=0}^{m} \sum_{a=0}^{2m+1-2c} q^{\binom{2c}{2} -a(2m+1-2c-a)} \Y^{(a)}  \qbinom{\kk; 1-m}{c}  F^{(2m+1-2c-a)}
\label{dvp2m+1}
\\
& \qquad
+\sum_{c=0}^m \sum_{a=0}^{2m-2c} q^{\binom{2c+1}{2}-2m -a(2m-2c-a)} \Y^{(a)}  \qbinom{\kk; 1-m}{c} K^{-1}  F^{(2m-2c-a)}. 
\notag
\end{align}
\end{thm}
The proof of this theorem will be given in \S \ref{sec:proof:iDPK} below.
It follows from Theorem~\ref{thm:iDPK} that $\dvp{n} \in \BA$ for all $n$.

\begin{example}
Recall $b^{(n)}$ from \eqref{eq:divB}. Here are some examples of $\dvp{n}$, for $2\le n\le 4$:
\begin{align*}
%\dvp{1} 
%&= b^{(1)} + K^{-1},
%\\
\dvp{2} 
&= b^{(2)} +q^{-1}K^{-1}F +q^{-1}\Y K^{-1}  + q^3 [\kk;0],
\\
 \dvp{3} 
&=  b^{(3)} + q \Y [\kk;0] + q [\kk;0] F 
+ q^{-2} \Y^{(2)} K^{-1} + q^{-3} \Y K^{-1} F + q^{-2} K^{-1} F^{(2)} + q [\kk;0] K^{-1},
\\
 \dvp{4} 
 & = b^{(4)} + q^3 \Y^{(2)} [\kk;-1] + q^2 \Y [\kk;-1] F + q^3 [\kk;-1] F^{(2)} + q^{10} \qbinom{\kk;-1}{2}    
 \\ &  
\qquad + q^{-3} \Y^{(3)} K^{-1}
+ q^{-5} \Y^{(2)} K^{-1} F + q^{-5} \Y K^{-1} F^{(2)} + q^{-3} K^{-1} F^{(3)}
\\
&\qquad + q^2 \Y [\kk;-1] K^{-1} + q^2 [\kk;-1] K^{-1}F.
\end{align*}
\end{example}

\subsection{Some reformulations}

The next $F\kk\Y$-formulae follow easily 
from the formulae in Theorem~\ref{thm:iDPK} by applying the anti-involution $\vs$ in Lemma~\ref{lem:anti}.
The proof is skipped.

\begin{prop}
  \label{iDPK:FkY}
For $m\ge 0$, we have 
\begin{align*}
\dvp{2m} &= \sum_{c=0}^m \sum_{a=0}^{2m-2c} (-1)^c q^{c+ a(2m-2c-a)} F^{(a)}  \qbinom{\kk;m-c}{c}  \Y^{(2m-2c-a)}
\\ 
&\qquad 
+ \sum_{c=0}^{m-1} \sum_{a=0}^{2m-1-2c} (-1)^c q^{-c-1+ 2m +a(2m-1-2c-a)} F^{(a)}  \qbinom{\kk;m-c}{c}K^{-1}  \Y^{(2m-1-2c-a)}, 
\\%
\dvp{2m+1} &=  \sum_{c=0}^{m} \sum_{a=0}^{2m+1-2c} (-1)^c q^{3c+ a(2m+1-2c-a)} F^{(a)}  \qbinom{\kk; m-c}{c}  \Y^{(2m+1-2c-a)}
\\
& \qquad
+\sum_{c=0}^m \sum_{a=0}^{2m-2c} (-1)^c q^{c+ 2m +a(2m-2c-a)} F^{(a)}  \qbinom{\kk; m-c}{c} K^{-1}  \Y^{(2m-2c-a)}. 
\end{align*}
\end{prop}

Below is a modified reformulation of Theorem~\ref{thm:iDPK}.

\begin{prop}
  \label{prop:iDPKdot}
For $m\ge 0$ and $\la \in\Z$, we have 
\begin{align*}
&\dvp{2m}  \one_{2\la} 
\\
&\quad = \sum_{c=0}^m \sum_{a=0}^{2m-2c} q^{2(a+c)(m-a-\la)-2ac-\binom{2c}{2}} \qbinom{m-c-a-\la}{c}_{q^2} E^{(a)}  F^{(2m-2c-a)}  \one_{2\la}
\\ 
&\qquad 
+ \sum_{c=0}^{m-1} \sum_{a=0}^{2m-1-2c} q^{2(a+c+1)(m-a-\la) -2ac-a-\binom{2c+2}{2}} 
\qbinom{m-c-a-\la-1}{c}_{q^2} E^{(a)} F^{(2m-1-2c-a)}  \one_{2\la}; 
\\
& \dvp{2m+1}  \one_{2\la} 
\\
&\quad =  \sum_{c=0}^{m} \sum_{a=0}^{2m+1-2c} 
q^{2(a+c)(m-a-\la)-2ac+a-\binom{2c}{2}} \qbinom{m-c-a-\la+1}{c}_{q^2} E^{(a)} F^{(2m+1-2c-a)}  \one_{2\la}
\\
& \qquad
+\sum_{c=0}^m \sum_{a=0}^{2m-2c} 
q^{2(a+c+1)(m-a-\la)-2ac-\binom{2c+2}{2}+1} \qbinom{m-c-a-\la}{c}_{q^2} E^{(a)} F^{(2m-2c-a)}  \one_{2\la}. 
\end{align*}  
In particular, we have $\dvp{n}\one_{2\la} \in \UAdot_\ev$, for all $n\in \N, \la \in\Z$. 
\end{prop}

%\begin{align*}
%&\dvp{2m}  \one_{2\la} \\
%&\quad = \sum_{c=0}^m \sum_{a=0}^{2m-2c} q^{-\binom{2c}{2} -a(2m-a) +2c(m-\la)} \qbinom{m-c-a-\la}{c}_{q^2} \Y^{(a)}  F^{(2m-2c-a)}  \one_{2\la} \\ 
%&\qquad 
%+ \sum_{c=0}^{m-1} \sum_{a=0}^{2m-1-2c} q^{-\binom{2c+2}{2} -a(2m-a+1) +(2c+2)(m-\la)} \qbinom{m-c-a-\la-1}{c}_{q^2} \Y^{(a)} F^{(2m-1-2c-a)}  \one_{2\la}; 
%% \\
%& \dvp{2m+1}  \one_{2\la}  \\
%&\quad =  \sum_{c=0}^{m} \sum_{a=0}^{2m+1-2c} q^{-\binom{2c}{2} -a(2m-a+1) +2c(m-\la)} \qbinom{m-c-a-\la+1}{c}_{q^2}  \Y^{(a)} F^{(2m+1-2c-a)}  \one_{2\la} \\
%& \qquad
%+\sum_{c=0}^m \sum_{a=0}^{2m-2c} q^{-\binom{2c+2}{2}+1 -a(2m-a+2) +(2c+2)(m-\la)} \qbinom{m-c-a-\la}{c}_{q^2} \Y^{(a)} F^{(2m-2c-a)}  \one_{2\la}. 
%\end{align*}  

\begin{proof}
Note $F^{(2m-2c-a)}  \one_{2\la} = \one_{2(\la-2m+2c+a)} F^{(2m-2c-a)}$, cf. \eqref{dvp2m}. 
Using \eqref{k=qbinom} and Lemma~\ref{dpY}, we compute
\begin{align*}
&q^{\binom{2c+1}{2} -a(2m-2c-a)}
\Y^{(a)} \qbinom{\kk;1-m}{c}  \one_{2(\la-2m+2c+a)}
\\
&= q^{\binom{2c+1}{2} -a(2m-2c-a)} 
q^{2a(2m-\la-2c-a) -a^2}
q^{2c(m-2c-a-\la)} \qbinom{m-c-a-\la}{c}_{q^2} E^{(a)} \one_{2(\la-2m+2c+a)}
 \\
&= q^{2(a+c)(m-a-\la)-2ac-\binom{2c}{2}} \qbinom{m-c-a-\la}{c}_{q^2} E^{(a)} \one_{2(\la-2m+2c+a)}.
\end{align*}
Similarly  we have
\begin{align*}
&q^{\binom{2c+2}{2}-2m -a(2m-1-2c-a)} \Y^{(a)} \qbinom{\kk;1-m}{c}K^{-1}  \one_{2(\la-2m+2c+a+1)}
\\
&= q^{\binom{2c+2}{2}-2m -a(2m-1-2c-a)} 
q^{-2(a+1)(\la-2m+2c+a+1) -a^2} \cdot
\\
&\qquad \quad q^{2c(m-2c-a-\la-1)} \qbinom{m-c-a-\la-1}{c}_{q^2} E^{(a)} \one_{2(\la-2m+2c+a+1)}
 \\
&= q^{2(a+c+1)(m-a-\la) -2ac-a-\binom{2c+2}{2}} \qbinom{m-c-a-\la-1}{c}_{q^2} E^{(a)} \one_{2(\la-2m+2c+a+1)}.
\end{align*}
This establishes the first formula.

Similarly, the second formula follows from \eqref{dvp2m+1},  \eqref{k=qbinom}  and the following computations:
\begin{align*}
&q^{\binom{2c}{2} -a(2m+1-2c-a)} \Y^{(a)} \qbinom{\kk; 1-m}{c} \one_{2(\la-2m+2c+a-1)}
\\
&\quad
= q^{\binom{2c}{2} -a(2m+1-2c-a)} 
q^{2a(2m-\la-2c-a+1)-a^2} \cdot
\\
&\qquad \qquad
q^{2c(m-2c-a-\la+1)} \qbinom{m-c-a-\la+1}{c}_{q^2} E^{(a)} \one_{2(\la-2m+2c+a-1)}
 \\
&\quad
= q^{2(a+c)(m-a-\la)-2ac+a-\binom{2c}{2}} \qbinom{m-c-a-\la+1}{c}_{q^2} E^{(a)} \one_{2(\la-2m+2c+a-1)};
\\ \\
&q^{\binom{2c+1}{2}-2m -a(2m-2c-a)} \Y^{(a)} \qbinom{\kk; 1-m}{c} K^{-1} \one_{2(\la-2m+2c+a)}
\\
&\quad
= q^{\binom{2c+1}{2}-2m -a(2m-2c-a)}  
q^{-2(a+1)(\la-2m+2c+a)-a^2}  \cdot
\\
&\qquad\qquad
q^{2c(m-2c-a-\la)} \qbinom{m-c-a-\la}{c}_{q^2} E^{(a)} \one_{2(\la-2m+2c+a)}
 \\
&\quad
= q^{2(a+c+1)(m-a-\la)-2ac-\binom{2c+2}{2}+1} \qbinom{m-c-a-\la}{c}_{q^2} E^{(a)} \one_{2(\la-2m+2c+a)}.
\end{align*}
The proposition is proved.
\end{proof}

\subsection{The $\imath$-canonical basis for simple $\U$-modules $L(2\la)$}
  \label{sec:iCB:evK}
  
Recall $\cbinom{m}{c}$ from \eqref{cbinom}. 
The following theorem confirms \cite[Conjecture~4.13]{BW13}.
 \begin{thm}  \label{thm:iCB:evK}
  $\quad$
  \begin{enumerate}
  \item
  The set $\{\dvp{n} \vev \mid 0\le n \le 2\la \}$ forms the $\imath$-canonical basis for $L(2\la)$, for each $\la \in \N$. 
  Moreover, $\dvp{2\la+1}\vev =\dvp{2\la} \vev$, and $\dvp{n} \vev =0$ for $n \ge 2\la+2$.

\item
The set $\{\dvp{n} \mid n \in \N \}$ forms the $\imath$-canonical basis for $\Ui$.  %\dot{\U}^\imath_{\ev}$. 
\end{enumerate}
 \end{thm}
 
 \begin{proof}
Let $\la, m \in \N$.  We compute by Proposition~\ref{iDPK:FkY} that 
\begin{align} 
\dvp{2m} \vev 
%&= \sum_{c=0}^m (-1)^c q^{c} F^{(2m-2c)}  \qbinom{\kk;m-c}{c} \vev
%\notag\\&\quad  + \sum_{c=0}^{m-1} (-1)^c q^{-c-1+2m} F^{(2m-1-2c)}  \qbinom{\kk;m-c}{c} K^{-1} \vev \notag \\
& = \sum_{c=0}^{m} q^{-2c^2-c} \cbinom{m-\la-c}{c} F^{(2m-2c)}  \vev
  \label{dvp:2m-mod}  \\&\quad + \sum_{c=0}^{m-1} q^{-2c^2-3c-1+2m-2\la} \cbinom{m-\la-c}{c} F^{(2m-1-2c)}  \vev;
\notag \\
\dvp{2m+1} \vev  
%&= \sum_{c=0}^{m} (-1)^c q^{3c} F^{(2m+1-2c)}  \qbinom{\kk;m-c}{c} \vev 
%\notag \\&\quad +  \sum_{c=0}^{m} (-1)^c q^{c+2m}  F^{(2m-2c)}  \qbinom{\kk;m-c}{c}K^{-1} \vev  \notag \\
&= \sum_{c=0}^m q^{-2c^2+c} \cbinom{m-\la-c}{c} F^{(2m+1-2c)}  \vev
  \label{dvp:2m+1-mod} 
  \\&\quad + \sum_{c=0}^m q^{-2c^2-c+2m-2\la} \cbinom{m-\la-c}{c} F^{(2m-2c)}  \vev.
 \notag
\end{align}
We observe that 
\begin{align}  \label{evK=0}
\begin{split}
    F^{(2m-1-2c)} \vev =F^{(2m-2c)}  \vev = F^{(2m+1-2c)} \vev =0, & \quad \text{ if } c<m-\la.
  \end{split}
\end{align}
It follows from \eqref{cbin=0}, \eqref{dvp:2m-mod}, \eqref{dvp:2m+1-mod} and \eqref{evK=0}  
that $\dvp{2m} \vev =\dvp{2m+1} \vev =0$, for $m\ge \la+1$;
moreover, for $m\le \la$, we have
 \begin{align*}
 \dvp{2m} \vev  & \in F^{(2m)}  \vev + \sum_{c\ge 1}q^{-1} \N[q^{-1}] F^{(2m-2c)} \vev,
 \\
 \dvp{2m-1} \vev  & \in F^{(2m-1)}  \vev + \sum_{c\ge 1}q^{-1} \N[q^{-1}] F^{(2m-1-2c)} \vev,
  \\
 \dvp{2\la+1} \vev  & = \dvp{2\la}  \vev.
 \end{align*}
 Therefore, the first statement follows by the characterization properties of the $\imath$-canonical basis for $L(2\la)$.
 The second statement follows now from the definition of the $\imath$-canonical basis for $\Ui$  %\dot{\U}^\imath_{\ev}$ 
 using the projective system $\{ L(2\la) \}_{\la\in \N}$; cf. \cite[\S6]{BW16} (also cf. \cite[\S4]{BW13}). 
 \end{proof}

\subsection{Proof of Theorem~\ref{thm:iDPK}}
  \label{sec:proof:iDPK}

We prove by induction on $n$, where the  base cases for $\dvp{n}$ with $n=0,1,2$ are clear. The induction is carried out in 2 steps. 

\vspace{.3cm}
(1) Assuming the formula \eqref{dvp2m} for $\dvp{2m}$, we shall establish the formula \eqref{dvp2m+1} for $\dvp{2m+1}$.

We shall make a repeated use of the following formula, which easily follows from \eqref{FYn} :
\begin{equation}
 \label{sYa}
 (\Y+F+K^{-1}) \Y^{(a)}
 = [a+1]\Y^{(a+1)} +q^{-2a} \Y^{(a)}F +\Y^{(a-1)} \frac{q^{3-3a}K^{-2} -q^{1-a}}{q^2-1} +q^{-2a} \Y^{(a)}K^{-1}.
\end{equation} 

Let us denote the 2 summands for $\dvp{2m}$ in \eqref{dvp2m} by $S_0, S_1$, and so 
\[
\dvp{2m} =S_0 +S_1.
\]
Using \eqref{sYa} we can rewrite $t \cdot S_0$ in the $\Y \kk F$ form as 
\[
t \cdot S_0 =A_0 +A_1, 
\]
where
\begin{align*}  %\label{ss2m}
A_0
&=  \sum_{c=0}^m \sum_{a=0}^{2m-2c} q^{\binom{2c+1}{2} -a(2m-2c-a)}
\cdot 
\notag \\
& \left( [a+1] \Y^{(a+1)}  \qbinom{\kk; 1-m}{c}  F^{(2m-2c-a)}
 +q^{-2a} [2m+1-2c-a] \Y^{(a)}  \qbinom{\kk; -m}{c}  F^{(2m+1-2c-a)} \right.
 \notag \\
& \qquad +
\left. \Y^{(a-1)} \frac{q^{3-3a} K^{-2} - q^{1-a} }{q^2-1} \qbinom{\kk; 1-m}{c}  F^{(2m-2c-a)} \right),
\notag
\\
A_1
&= \sum_{c=0}^m \sum_{a=0}^{2m-2c} q^{\binom{2c+1}{2} -a(2m-2c-a) -2a} 
\Y^{(a)}  \qbinom{\kk; 1-m}{c} K^{-1} F^{(2m-2c-a)}.
\end{align*}

Using \eqref{sYa} we can also rewrite $t \cdot S_1$ in the $\Y \kk F$ form as 
\[
t \cdot S_1 =B_1 +B_0, 
\]
where
\begin{align*} 
B_1
&=  \sum_{c=0}^{m-1} \sum_{a=0}^{2m-1-2c} q^{\binom{2c+2}{2}-2m -a(2m-1-2c-a)}
\cdot  \\
& \qquad \quad \left( [a+1] \Y^{(a+1)}  \qbinom{\kk; 1-m}{c} K^{-1} F^{(2m-1-2c-a)}
\right.
\\
& \qquad \qquad +q^{-2a-2} [2m-2c-a] \Y^{(a)}  \qbinom{\kk; -m}{c} K^{-1} F^{(2m-2c-a)}  
\\
& \qquad \qquad +
\left. \Y^{(a-1)} \frac{q^{3-3a} K^{-2} - q^{1-a} }{q^2-1} \qbinom{\kk; 1-m}{c} K^{-1} F^{(2m-1-2c-a)} \right),
\\
B_0
&= \sum_{c=0}^{m-1} \sum_{a=0}^{2m-1-2c} q^{\binom{2c+2}{2}-2m -a(2m-1-2c-a) -2a} 
\Y^{(a)}  \qbinom{\kk; 1-m}{c} K^{-2} F^{(2m-1-2c-a)}.
\end{align*}
Hence by \eqref{eq:ttodd2} we have 
\begin{equation}   \label{eq:dd}
[2m+1] \dvp{2m+1} = t \cdot \dvp{2m} =(A_0 +B_0) +(A_1+B_1).
\end{equation}
We shall rewrite $A_0+B_0$ in the form
\[
A_0 +B_0 =  \sum_{c=0}^m \sum_{a=0}^{2m+1-2c} \Y^{(a)} f_{a,c}^0(\kk) F^{(2m+1-2c-a)}, 
\]
where
\begin{align*}
f_{a,c}^0(\kk) &= q^{\binom{2c+1}{2} -(a-1)(2m+1-2c-a)}   [a]\,  \qbinom{\kk; 1-m}{c} \\
&\quad + \left( q^{\binom{2c+1}{2} -a(2m-2c-a) -2a} [2m+1-2c-a]\, \qbinom{\kk; -m}{c} \right. \\
&\qquad\quad \left. +q^{\binom{2c-1}{2} -(a+1)(2m+1-2c-a)}  \frac{q^{-3a} K^{-2} - q^{-a} }{q^2-1} \qbinom{\kk; 1-m}{c-1} \right.
\\
&\qquad\quad \left. + q^{\binom{2c}{2}-2m -a(2m+1-2c-a) -2a}  \qbinom{\kk; 1-m}{c-1} K^{-2} \right). 
\end{align*}
A direct computation shows that
\begin{align*}
& f_{a,c}^0(\kk) 
\\
&= q^{\binom{2c}{2} -a(2m+1-2c-a)} \cdot q^{2m+1-a} [a] \,  \qbinom{\kk; 1-m}{c} \\
&\qquad + q^{\binom{2c}{2} -a(2m+1-2c-a)}  \qbinom{\kk; 1-m}{c-1}  \cdot
\\
&\qquad\qquad 
 \left( q^{2c-a} [2m+1-2c-a] \frac{q^{-4m}K^{-2}-1}{q^{4c}-1}
 +q^{a-2m}  \frac{q^{-3a} K^{-2} - q^{-a} }{q^2-1}  + q^{-2m -2a} K^{-2}
 \right)  
\\%
&= q^{\binom{2c}{2} -a(2m+1-2c-a)} \left( q^{2m+1-a} [a]  \qbinom{\kk; 1-m}{c}  +  \qbinom{\kk; 1-m}{c-1}  q^{-a} [2m+1-a]  \frac{q^{4c-4m}K^{-2}-1}{q^{4c}-1} \right)
\\%
&=q^{\binom{2c}{2} -a(2m+1-2c-a)} [2m+1] \qbinom{\kk; 1-m}{c}.
\end{align*}

On the other hand, we shall rewrite $A_1+B_1$ in the form
\[
A_1 +B_1 =  \sum_{c=0}^m \sum_{a=0}^{2m-2c} \Y^{(a)} f_{a,c}^1(\kk) K^{-1} F^{(2m-2c-a)}, 
\]
where
\begin{align*}
f_{a,c}^1(\kk) &= q^{\binom{2c+1}{2} -a(2m-2c-a) -2a}  \qbinom{\kk; 1-m}{c} 
+  q^{\binom{2c+2}{2} -2m-(a-1)(2m-2c-a)} [a]\, \qbinom{\kk; 1-m}{c} \\
&\qquad\quad  + q^{\binom{2c+2}{2}-2m -a(2m-1-2c-a) -2a-2} [2m-2c-a] \qbinom{\kk; -m}{c}  
\\
&\qquad\quad +q^{\binom{2c}{2} -2m -(a+1)(2m-2c-a)}  \frac{q^{-3a} K^{-2} - q^{-a} }{q^2-1} \qbinom{\kk; 1-m}{c-1}.
\end{align*}
Denote by $U_1$ the sum of the first two summands 
and $U_2$ the sum of the last 2 summands of $f_{a,c}^1(\kk)$ above, so that 
\[
f_{a,c}^1(\kk)=U_1+U_2.
\]
We have 
\begin{align*}
U_1&=q^{\binom{2c+1}{2}-2m -a(2m-2c-a)}
\left( 
q^{2m-2a}  + q^{1+2m-a} [a]  
 \right) \qbinom{\kk; 1-m}{c}
 \\
 &= q^{\binom{2c+1}{2}-2m -a(2m-2c-a)}  \cdot   q^{2m-a}[a+1]  \qbinom{\kk; 1-m}{c}.
\end{align*}
Moreover, a direct computation shows that $U_2$ is equal to
\begin{align*}
%&=   q^{\binom{2c+1}{2}-2m -a(2m-2c-a)}  \cdot
%\left(
% q^{2c-a-1} [2m-2c-a] \qbinom{\kk; -m}{c} 
%+ q^{-2m+a} \frac{q^{-3a} K^{-2} - q^{-a} }{q^2-1} \qbinom{\kk; 1-m}{c-1}
%\right)
%\\
%%
&=  q^{\binom{2c+1}{2}-2m -a(2m-2c-a)}  \qbinom{\kk; 1-m}{c-1}  \cdot
\\
& \qquad 
\left(
q^{2c-a-1} [2m-2c-a] \frac{q^{-4m}K^{-2}-1}{q^{4c}-1}
+ q^{-2m+a} \frac{q^{-3a} K^{-2} - q^{-a} }{q^2-1} 
\right)
\\
& =q^{\binom{2c+1}{2}-2m -a(2m-2c-a)}  \qbinom{\kk; 1-m}{c-1} \cdot 
\left(
q^{-a-1} [2m-a] \frac{q^{4c-4m}K^{-2}-1}{q^{4c}-1}
\right)
\\
&= q^{\binom{2c+1}{2}-2m -a(2m-2c-a)} \cdot q^{-a-1} [2m-a]  \qbinom{\kk; 1-m}{c}.
\end{align*}
Hence we conclude that 
\[
f_{a,c}^1(\kk)=U_1+U_2 = q^{\binom{2c+1}{2}-2m -a(2m-2c-a)} [2m+1]  \qbinom{\kk; 1-m}{c}.
\]
The formula \eqref{dvp2m+1} follows from \eqref{eq:dd} and
the formulae of $f_{a,c}^0(\kk)$ and $f_{a,c}^1(\kk)$ above.

\vspace{.3cm}

(2) Now we shall prove the formula  for $\dvp{2m+2}$ (replacing $m$ by $m+1$ in \eqref{dvp2m})
by assuming the formula \eqref{dvp2m+1} for $\dvp{2m+1}$.

Let us denote the 2 summands for $\dvp{2m+1}$ in \eqref{dvp2m+1} by $T_0, T_1$, and so 
\[
\dvp{2m+1} =T_0 +T_1.
\]
Using \eqref{sYa} we can rewrite $t \cdot T_0$ in the $\Y \kk F$ form as 
\[
t \cdot T_0 =C_0 +C_1, 
\]
where
\begin{align*}  %\label{ss2m}
C_0
&=  \sum_{c=0}^m \sum_{a=0}^{2m+1-2c} q^{\binom{2c}{2} -a(2m+1-2c-a)}
\cdot 
\notag \\
& \left( [a+1] \Y^{(a+1)}  \qbinom{\kk; 1-m}{c}  F^{(2m+1-2c-a)}
 +q^{-2a} [2m+2-2c-a] \Y^{(a)}  \qbinom{\kk; -m}{c}  F^{(2m+2-2c-a)} \right.
 \notag \\
& \qquad +
\left. \Y^{(a-1)} \frac{q^{3-3a} K^{-2} - q^{1-a} }{q^2-1} \qbinom{\kk; 1-m}{c}  F^{(2m+1-2c-a)} \right),
\notag
\\
C_1
&= \sum_{c=0}^m \sum_{a=0}^{2m+1-2c} q^{\binom{2c}{2} -a(2m+1-2c-a) -2a} 
\Y^{(a)}  \qbinom{\kk; 1-m}{c} K^{-1} F^{(2m+1-2c-a)}.
\end{align*}

Using \eqref{sYa} we can also rewrite $t \cdot T_1$ in the $\Y \kk F$ form as 
\[
t \cdot T_1 =D_1 +D_0, 
\]
where
\begin{align*} 
D_1
&=  \sum_{c=0}^{m} \sum_{a=0}^{2m-2c} q^{\binom{2c+1}{2}-2m -a(2m-2c-a)}
\cdot 
\notag \\
&\qquad  \left( [a+1] \Y^{(a+1)}  \qbinom{\kk; 1-m}{c} K^{-1} F^{(2m-2c-a)} \right.
\\
&\qquad\qquad +q^{-2a-2} [2m+1-2c-a] \Y^{(a)}  \qbinom{\kk; -m}{c} K^{-1} F^{(2m+1-2c-a)}  
 \notag \\
& \qquad\qquad 
 + \left. \Y^{(a-1)} \frac{q^{3-3a} K^{-2} - q^{1-a} }{q^2-1} \qbinom{\kk; 1-m}{c} K^{-1} F^{(2m-2c-a)} \right),
\notag
\\
D_0
&= \sum_{c=0}^{m} \sum_{a=0}^{2m-2c} q^{\binom{2c+1}{2}-2m -a(2m-2c-a) -2a} 
\Y^{(a)}  \qbinom{\kk; 1-m}{c} K^{-2} F^{(2m-2c-a)}.
\end{align*}

We denote the two summands in \eqref{dvp2m} by $G_0$ and $G_1$, so that 
\[
\dvp{2m} =G_0 +G_1.
\] 
Hence by \eqref{eq:ttodd2} we have 
\begin{align}
  \label{eq:CDG01}
\begin{split}
[2m+2] \dvp{2m+2} &= t \cdot \dvp{2m+1} -[2m+1] \dvp{2m} 
\\
&=(C_0 +D_0 -[2m+1]G_0) +(C_1+D_1 -[2m+1]G_1).
\end{split}
\end{align} 
We shall write $C_0+D_0 -[2m+1]G_0$ in the form
\[
C_0 +D_0 -[2m+1]G_0 =  \sum_{c=0}^{m+1} \sum_{a=0}^{2m+2-2c} \Y^{(a)} g_{a,c}^0(\kk) F^{(2m+2-2c-a)}.
\]
Indeed $g_{a,c}^0 (\kk)$ can be organized (by separating the second summand in $C_0$) as
\begin{align}
  \label{gac0}
g_{a,c}^0 (\kk) &= q^{\binom{2c}{2} - a(2m+1-2c-a) -2a}   [2m+2-2c-a]\, \qbinom{\kk; -m}{c}   +  
X_0, 
\end{align}
where  
\begin{align*}
X_0 &= q^{\binom{2c}{2} -(a-1)(2m+2-2c-a)}   [a]\,  \qbinom{\kk; 1-m}{c} \\
&\quad   +q^{\binom{2c-2}{2} -(a+1)(2m+2-2c-a)}  \frac{q^{-3a} K^{-2} - q^{-a} }{q^2-1} \qbinom{\kk; 1-m}{c-1} 
\\
&\quad +  q^{\binom{2c-1}{2} -2m -a(2m+2-2c-a) -2a}  \qbinom{\kk; 1-m}{c-1}  K^{-2} \\
&\quad   -q^{\binom{2c-1}{2} - a(2m+2-2c-a)} [2m+1] \qbinom{\kk; 1-m}{c-1}.
\end{align*}
We  rewrite $X_0 = q^{\binom{2c+1}{2} - a(2m+2-2c-a)}  \qbinom{\kk; 1-m}{c-1} Z_0$,
where
\begin{align*}
Z_0 &= q^{2m-a-4c+2} [a] \frac{q^{4c-4m} K^{-2} - 1}{q^{4c}-1} 
+q^{1-2m+a-4c} \frac{q^{-3a} K^{-2} - q^{-a} }{q^2-1} 
\\
&\quad + q^{1-2m-4c-2a} K^{-2} 
-q^{1-4c} [2m+1]. 
\end{align*}
A direct computation shows that 
\[
Z_0 =q^{2m+2-2c-a} [2c+a] \frac{q^{-4m} K^{-2} - 1}{q^{4c}-1},
\]
and this gives us
\begin{align*}
X_0 =q^{\binom{2c+1}{2} - a(2m+2-2c-a) -2a} q^{2m+2-2c-a}  [2c+a]\, \qbinom{\kk; -m}{c}.
\end{align*}
Plugging this formula for $X_0$ into \eqref{gac0} we obtain
\begin{align}
\label{gac0:answer}
g_{a,c}^0 (\kk) =q^{\binom{2c+1}{2} - a(2m+2-2c-a) -2a}  [2m+2]\, \qbinom{\kk; -m}{c}.
\end{align}

Now let us simplify  $C_1+D_1 -[2m+1]G_1$ by writing it in the form
\[
C_1 +D_1 -[2m+1]G_1 =  \sum_{c=0}^{m} \sum_{a=0}^{2m+1-2c} \Y^{(a)} g_{a,c}^1(\kk) K^{-1} F^{(2m+1-2c-a)}.
\]
Indeed $g_{a,c}^1 (\kk)$ can be organized (by pulling out the second summand of $D_1$) as
\begin{align}
  \label{gac1}
g_{a,c}^1 (\kk) &=  q^{\binom{2c+1}{2}-2m -a(2m-2c-a) -2a-2} [2m+1-2c-a]   \qbinom{\kk; -m}{c}  + X_1, 
\end{align}
where  
\begin{align*}
X_1 &=
q^{\binom{2c}{2} -a(2m+1-2c-a) -2a}  \qbinom{\kk; 1-m}{c} 
\\
&\qquad + q^{\binom{2c+1}{2}-2m - (a-1)(2m+1-2c-a)} [a]   \qbinom{\kk; 1-m}{c}   
\\
& \qquad + q^{\binom{2c-1}{2}-2m - (a+1)(2m+1-2c-a)}   \frac{q^{-3a} K^{-2} - q^{-a} }{q^2-1} \qbinom{\kk; 1-m}{c-1}
 \\
&\quad   -q^{\binom{2c}{2} -2m - a(2m+1-2c-a)} [2m+1] \qbinom{\kk; 1-m}{c-1}.
\end{align*}
We  rewrite 
\[
X_1 = q^{\binom{2c+2}{2} -2m-2 - a(2m+1-2c-a)}  \qbinom{\kk; 1-m}{c-1} Z_1,
\]
where
\begin{align*}
Z_1 &= q^{2m-2a-4c+1} \frac{q^{4c-4m} K^{-2} - 1}{q^{4c}-1} 
+q^{2+2m-a-4c} [a] \frac{q^{4c-4m} K^{-2} - 1}{q^{4c}-1} 
\\
&\qquad+q^{1-2m+a-4c} \frac{q^{-3a} K^{-2} - q^{-a} }{q^2-1} 
 -q^{1-4c} [2m+1]. 
\end{align*}
A direct computation shows that
\[
Z_1 = q^{2m-2c-a+1} [2c+a+1] \frac{q^{-4m} K^{-2} - 1}{q^{4c}-1},
\]
and this gives us
\begin{align*}
X_1 = q^{\binom{2c+2}{2} -2m-2 - a(2m+1-2c-a)} q^{2m-2c-a+1} [2c+a+1] \, \qbinom{\kk; -m}{c}.
\end{align*}
Plugging this formula for $X_1$ into \eqref{gac1} we obtain
\begin{align}
\label{gac1:answer}
g_{a,c}^1 (\kk) & =q^{\binom{2c+2}{2} -2m-2 - a(2m+1-2c-a)}  [2m+2]\, \qbinom{\kk; -m}{c}.
\end{align}
The formula for $\dvp{2m+2}$ now follows from \eqref{eq:CDG01}, \eqref{gac0:answer} and \eqref{gac1:answer}.

This completes the proof of Theorem~\ref{thm:iDPK}.

\section{Formulae for $\imath$-divided powers $\dvd{n}$ and $\imath$-canonical basis on $L(2\la+1)$}
  \label{sec:dvK:odd}

In this section we formulate a variant of $\imath$-divided powers (of odd weights) $\dvd{n}$ starting with $t$ in \eqref{def:t}, and established closed formulae for $\dvd{n}$ in $\U$.

\subsection{The $\imath$-divided powers $\dvd{n}$}

Recall $t=\Y+F+K^{-1}.$ Set $\dvd{0}=1, \dvd{1}= t$. 
The divided powers $\dvd{n}$ (of odd weights) satisfy the recursive relations: 
\begin{align}  
\label{eq:ttodd2:dvd} 
\begin{split}
t \cdot \dvd{2a-1} &=  [2a] \dvd{2a},
\\
t \cdot \dvd{2a} &=[2a+1] \dvd{2a+1}+[2a] \dvd{2a-1}, \quad \text{ for } a \ge 1.
\end{split}
\end{align}
Note  $\dvd{2} =t^2/[2]$, and $\dvd{3} = t(t^2-[2]^2)/[3]!$. 
Note that $\dvd{n}$  satisfy the same recursive relations \eqref{eq:tt} as $\dvev{n}$.
Hence $\dvd{n}$ is formally given by the same formula as in \eqref{def:idp} (with $\ttt$ therein replaced by $t$). 

\subsection{The $\Y \kk F$-formula for $\dvd{n}$}

Recall $\LR{\kk;a}{n}$ from \eqref{brace}. 

\begin{thm}
  \label{thm:iDPK:odd}
For $m\ge 0$, we have 
\begin{align}
\dvd{2m} &= \sum_{c=0}^m \sum_{a=0}^{2m-2c} q^{\binom{2c-1}{2}-1 -a(2m-2c-a)}
\Y^{(a)}  \LR{\kk;2-m}{c}  F^{(2m-2c-a)}
\label{dvd2m}
\\ 
&\quad 
+ \sum_{c=0}^{m-1} \sum_{a=0}^{2m-1-2c} q^{\binom{2c}{2}+1-2m  -a(2m-1-2c-a)} \Y^{(a)}  \LR{\kk;2-m}{c} K^{-1}  F^{(2m-1-2c-a)}, 
\notag 
\\
\notag\\
\dvd{2m+1} &=  \sum_{c=0}^{m} \sum_{a=0}^{2m+1-2c} q^{\binom{2c}{2} -a(2m+1-2c-a)} \Y^{(a)} \LR{\kk;1-m}{c}  F^{(2m+1-2c-a)}
\label{dvd2m+1}
\\
& \quad
+\sum_{c=0}^{m} \sum_{a=0}^{2m-2c} q^{\binom{2c+1}{2}-2m -a(2m-2c-a)} \Y^{(a)}  \LR{\kk;1-m}{c} K^{-1}  F^{(2m-2c-a)}. 
\notag
\end{align}
\end{thm}
The proof of this theorem will be given in \S \ref{sec:proof:iDPK:odd} below.
It follows from Theorem~\ref{thm:iDPK:odd} that $\dvp{n} \in \DA$ for all $n$. 

\begin{example}
We have the following examples of $\dvd{n}$, for $2\le n \le 4$: 
\begin{align*}
%\dvd{1}  &= b^{(1)} + K^{-1}, \\
\dvd{2} 
&= b^{(2)} + q^{-1} \llbracket \kk;1 \rrbracket +q^{-1}K^{-1}F +q^{-1}\Y K^{-1} ,
\\
\dvd{3} 
&= b^{(3)} + q \Y \llbracket \kk;0 \rrbracket + q \llbracket \kk;0 \rrbracket F 
+ q^{-2} \Y^{(2)} K^{-1} + q^{-3} \Y K^{-1} F + q^{-2} K^{-1} F^{(2)} + q \llbracket \kk;0 \rrbracket K^{-1},
\\
\dvd{4}
&= b^{(4)} + q^{-1} \Y^{(2)} \llbracket \kk;0 \rrbracket + q^{-2} \Y \llbracket \kk;0 \rrbracket F + q^{-1} \llbracket \kk;0 \rrbracket F^{(2)} + q^2 \LR{\kk;0}{2}
\\
&\quad + q^{-3} \Y^{(3)} K^{-1} + q^{-5} \Y^{(2)} K^{-1} F + q^{-5} \Y K^{-1} F^{(2)} + q^{-3} K^{-1} F^{(3)}
\\
& \quad + q^{-2} \Y \llbracket \kk;0 \rrbracket K^{-1} + q^{-2} \llbracket \kk;0 \rrbracket K^{-1} F.
%corrected: Y >>> \Y, twice
\end{align*}
\end{example}

\subsection{Some reformulations}

We obtain the following $F\kk\Y$-formula from Theorem~\ref{thm:iDPK:odd} by applying the anti-involution $\vs$ and Lemma~\ref{lem:anti2}.
We skip the proof, which is similar to the one for Proposition~\ref{iDP:FkY}. 
\begin{prop}   \label{iDPK:odd:FkY}
For $m\ge 0$, we have 
\begin{align*}
\dvd{2m} &= \sum_{c=0}^m \sum_{a=0}^{2m-2c} (-1)^c q^{c+ a(2m-2c-a)} F^{(a)}  \LR{\kk;m-c}{c}  \Y^{(2m-2c-a)}
\\ 
&\quad 
+ \sum_{c=0}^{m-1} \sum_{a=0}^{2m-1-2c} (-1)^c q^{-c-1+ 2m +a(2m-1-2c-a)} F^{(a)} \LR{\kk;m-c}{c} K^{-1}  \Y^{(2m-1-2c-a)}, 
\\
\\%
\dvd{2m+1} &=  \sum_{c=0}^{m} \sum_{a=0}^{2m+1-2c} (-1)^c q^{-c+ a(2m+1-2c-a)} F^{(a)}  \LR{\kk;1+m-c}{c}  \Y^{(2m+1-2c-a)}
\\
& \quad
+\sum_{c=0}^m \sum_{a=0}^{2m-2c} (-1)^c q^{-3c+ 2m +a(2m-2c-a)} F^{(a)}   \LR{\kk;1+m-c}{c} K^{-1}  \Y^{(2m-2c-a)}. 
\end{align*}
\end{prop}

Below is a modified reformulation of Theorem~\ref{thm:iDPK:odd}. Again the formulae below indicate that $\dvd{n} \one_{2\la-1}$
is a positive integral (i.e. $\N[q,q^{-1}]$-) linear combination of canonical basis elements of $\Udot$. 

\begin{prop}
For $m\ge 0$ and $\la \in \Z$, we have 
\begin{align*}
& \dvd{2m} \one_{2\la-1}
\\
&\quad = \sum_{c=0}^m \sum_{a=0}^{2m-2c} q^{2(a+c)(m-a-\la)-2ac +a-\binom{2c}{2}} \qbinom{m-c-a-\la+1}{c}_{q^2} E^{(a)}  F^{(2m-2c-a)}  \one_{2\la-1}
\\ 
&\qquad 
+ \sum_{c=0}^{m-1} \sum_{a=0}^{2m-1-2c} q^{2(a+c+1)(m-a-\la) -2ac-\binom{2c+2}{2}+1} 
\qbinom{m-c-a-\la}{c}_{q^2} E^{(a)} F^{(2m-1-2c-a)}  \one_{2\la-1}; 
\\
\\
& \dvd{2m+1} \one_{2\la-1}
\\
&\quad =  \sum_{c=0}^{m} \sum_{a=0}^{2m+1-2c} 
q^{2(a+c)(m-a-\la)-2ac+2a-\binom{2c-1}{2}+1} \qbinom{m-c-a-\la+1}{c}_{q^2} E^{(a)} F^{(2m+1-2c-a)}  \one_{2\la-1}
\\
& \qquad
+\sum_{c=0}^m \sum_{a=0}^{2m-2c} 
q^{2(a+c+1)(m-a-\la)-2ac+a-\binom{2c+1}{2}+1} \qbinom{m-c-a-\la}{c}_{q^2} E^{(a)} F^{(2m-2c-a)}  \one_{2\la-1}. 
\end{align*}  
In particular, we have $\dvd{n}\one_{2\la-1} \in \UAdot_\odd$, for all $n\in \N, \la \in\Z$. 
\end{prop}

\begin{proof}
Note $F^{(2m-2c-a)}  \one_{2\la-1} = \one_{2(\la-2m+2c+a)-1} F^{(2m-2c-a)}$, cf. \eqref{dvd2m}. 
Using \eqref{kqbinom2} and Lemma~\ref{dpY}, we compute
\begin{align*}
&q^{\binom{2c-1}{2} -1 -a(2m-2c-a)}
\Y^{(a)} \LR{\kk;2-m}{c}  \one_{2(\la-2m+2c+a)-1}
\\
&= q^{\binom{2c-1}{2} -1 -a(2m-2c-a)} 
q^{2a(2m-\la-2c-a) +a -a^2} \cdot
\\
&\qquad \quad
q^{2c(2+m-2c-a-\la)} \qbinom{m-c-a-\la+1}{c}_{q^2} E^{(a)} \one_{2(\la-2m+2c+a)-1}
 \\
&= q^{2(a+c)(m-a-\la)-2ac +a-\binom{2c}{2}} \qbinom{m-c-a-\la+1}{c}_{q^2} E^{(a)} \one_{2(\la-2m+2c+a)-1}.
\end{align*}
Similarly  we have
\begin{align*}
&q^{\binom{2c}{2} +1-2m -a(2m-1-2c-a)} \Y^{(a)} \LR{\kk;2-m}{c}K^{-1}  \one_{2(\la-2m+2c+a+1)-1}
\\
&= q^{\binom{2c}{2} +1-2m -a(2m-1-2c-a)} 
q^{-2(a+1)(\la-2m+2c+a+1) +a+1 -a^2} \cdot
\\
&\qquad \quad q^{2c(2+m-2c-a-\la-1)} \qbinom{m-c-a-\la}{c}_{q^2} E^{(a)} \one_{2(\la-2m+2c+a+1)-1}
 \\
&= q^{2(a+c+1)(m-a-\la) -2ac-\binom{2c+2}{2}+1} \qbinom{m-c-a-\la}{c}_{q^2} E^{(a)} \one_{2(\la-2m+2c+a+1)-1}.
\end{align*}
This establishes the first formula.

Similarly, the second formula follows from \eqref{dvd2m+1},  \eqref{kqbinom2}  and the following computations:
\begin{align*}
&q^{\binom{2c}{2} -a(2m+1-2c-a)} \Y^{(a)} \LR{\kk; 1-m}{c} \one_{2(\la-2m+2c+a-1)-1}
\\
&\quad
= q^{\binom{2c}{2} -a(2m+1-2c-a)} 
q^{2a(2m-\la-2c-a+1)+a-a^2} \cdot
\\
&\qquad \qquad
q^{2c(m-2c-a-\la+2)} \qbinom{m-c-a-\la+1}{c}_{q^2} E^{(a)} \one_{2(\la-2m+2c+a-1)-1}
 \\
&\quad
= q^{2(a+c)(m-a-\la)-2ac+2a-\binom{2c-1}{2}+1} \qbinom{m-c-a-\la+1}{c}_{q^2} E^{(a)} \one_{2(\la-2m+2c+a-1)-1};
\\ \\
&q^{\binom{2c+1}{2}-2m -a(2m-2c-a)} \Y^{(a)} \LR{\kk; 1-m}{c} K^{-1} \one_{2(\la-2m+2c+a)-1}
\\
&\quad
= q^{\binom{2c+1}{2}-2m -a(2m-2c-a)}  
q^{-2(a+1)(\la-2m+2c+a)+a+1-a^2}  \cdot
\\
&\qquad\qquad
q^{2c(m-2c-a-\la+1)} \qbinom{m-c-a-\la}{c}_{q^2} E^{(a)} \one_{2(\la-2m+2c+a)-1}
 \\
&\quad
= q^{2(a+c+1)(m-a-\la)-2ac+a-\binom{2c+1}{2}+1} \qbinom{m-c-a-\la}{c}_{q^2} E^{(a)} \one_{2(\la-2m+2c+a)-1}.
\end{align*}
The proposition is proved.
\end{proof}

\subsection{The $\imath$-canonical basis for simple $\U$-modules $L(2\la)$}
  \label{sec:iCB:evK}
  
Recall $\cbinom{m}{c}$ from \eqref{cbinom}. 
 The following theorem confirms \cite[Conjecture~4.13]{BW13}.
 \begin{thm}  \label{thm:iCB:oddK}
 $\quad$
  \begin{enumerate}
  \item
  The set $\{\dvd{n}\, \vodd \mid 0\le n \le 2\la +1 \}$ forms the $\imath$-canonical basis for $L(2\la+1)$, for each $\la \in \N$. 
  Moreover, $\dvd{2\la+2}\vodd =\dvd{2\la+1} \vodd$, and $\dvd{n}\, \vodd =0$ for $n \ge 2\la+3$.

\item
The set $\{\dvd{n} \mid n \in \N \}$ forms the $\imath$-canonical basis for $\Ui$.  %\dot{\U}^\imath_{\odd}$. 
\end{enumerate}
 \end{thm}
 
 \begin{proof}
Let $\la, m \in \N$.  We compute by Proposition~\ref{iDPK:odd:FkY} that 
\begin{align} 
\dvd{2m+2} \vodd 
%&= \sum_{c=0}^{m+1} (-1)^c q^{c} F^{(2m+2-2c)}  \LR{\kk;1+m-c}{c} \vodd
%\notag\\&\quad  + \sum_{c=0}^{m} (-1)^c q^{-c+1+2m} F^{(2m+1-2c)}  \LR{\kk;1+m-c}{c} K^{-1} \vodd \notag \\
& = \sum_{c=0}^{m+1} q^{-2c^2+c} \cbinom{m-\la-c}{c} F^{(2m+2-2c)}  \vodd
  \label{dvd:2m-mod}  
  \\&\quad + \sum_{c=0}^{m} q^{-2c^2-c+2m-2\la} \cbinom{m-\la-c}{c} F^{(2m+1-2c)}  \vodd;
\notag \\
\dvd{2m+1} \vodd  
&= \sum_{c=0}^m q^{-2c^2-c} \cbinom{m-\la-c}{c} F^{(2m+1-2c)}  \vodd
  \label{dvd:2m+1-mod} 
  \\&\quad + \sum_{c=0}^m q^{-2c^2-3c+2m-2\la-1} \cbinom{m-\la-c}{c} F^{(2m-2c)}  \vodd.
 \notag
\end{align}
We observe that 
\begin{align}  \label{oddK=0}
\begin{split}
  F^{(2m-2c)}  \vodd = F^{(2m+1-2c)}  \vodd = F^{(2m+2-2c)}  \vodd =0, & \quad \text{ if } c<m-\la. 
%  \\ \cbinom{m-\la-c}{c} =0, & \quad \text{ if } c\ge m-\la \text{ and } m \ge \la+1.
  \end{split}
\end{align}
It follows from \eqref{cbin=0}, \eqref{dvd:2m-mod}, \eqref{dvd:2m+1-mod} and \eqref{oddK=0}  
that $\dvd{2m} \vodd =\dvd{2m+1} \vodd =0$, for $m\ge \la+1$;
moreover, for $m\le \la$, we have
 \begin{align*}
 \dvd{2m} \vodd  & \in F^{(2m)}  \vodd + \sum_{c\ge 1}q^{-1} \N[q^{-1}] F^{(2m-2c)} \vodd,
 \\
 \dvd{2m+1} \vodd  & \in F^{(2m+1)}  \vodd + \sum_{c\ge 1}q^{-1} \N[q^{-1}] F^{(2m+1-2c)} \vodd,
  \\
 \dvd{2\la+2} \vodd  & = \dvd{2\la+1}  \vodd.
 \end{align*}
 Therefore, the first statement follows by the characterization properties of the $\imath$-canonical basis for $L(2\la)$.
 The second statement follows now from the definition of the $\imath$-canonical basis for $\Ui$ using
the projective system $\{ L(2\la) \}_{\la\in \N}$; cf. \cite[\S6]{BW16} (also cf. \cite[\S4]{BW13}). 
 \end{proof}

\subsection{Proof of Theorem~\ref{thm:iDPK:odd}}
  \label{sec:proof:iDPK:odd}

We prove by induction on $n$ for $\dvd{n}$, with the base cases for $n=0,1,2$ being clear.

\vspace{.3cm}
(1) We shall establish the formula \eqref{dvd2m} for $\dvd{2m}$ by assuming the formula for $\dvd{2m-1}$ below 
(which is obtained from \eqref{dvd2m+1} with $m$ replaced by $m-1$):
\begin{align}
 \label{dvd2m-1}
\dvd{2m-1} &=  \sum_{c=0}^{m-1} \sum_{a=0}^{2m-1-2c} q^{\binom{2c}{2} -a(2m-1-2c-a)} \Y^{(a)} \LR{\kk;2-m}{c}  F^{(2m-1-2c-a)}
\\
& \quad
+\sum_{c=0}^{m-1} \sum_{a=0}^{2m-2-2c} q^{\binom{2c+1}{2}-2m+2 -a(2m-2-2c-a)} \Y^{(a)}  \LR{\kk;2-m}{c} K^{-1}  F^{(2m-2-2c-a)}. 
\notag
\end{align}

Let us denote the 2 summands for $\dvd{2m-1}$ in \eqref{dvd2m-1} by $\texttt{S}_0, \texttt{S}_1$, and so 
\[
\dvd{2m-1} =\texttt{S}_0 +\texttt{S}_1.
\]
Using \eqref{sYa} we can rewrite $t \cdot \texttt{S}_0$ in the $\Y \kk F$ form as 
\[
t \cdot \texttt{S}_0 =\texttt{A}_0 +\texttt{A}_1, 
\]
where
\begin{align*}  
\texttt{A}_0
&=  \sum_{c=0}^{m-1} \sum_{a=0}^{2m-1-2c} q^{\binom{2c}{2} -a(2m-1-2c-a)} \cdot 
\notag \\
& \left( [a+1] \Y^{(a+1)}  \LR{\kk;2-m}{c}  F^{(2m-1-2c-a)}
 +q^{-2a} [2m-2c-a] \Y^{(a)}  \LR{\kk;1-m}{c}  F^{(2m-2c-a)} \right.
 \notag \\
& \qquad +
\left. \Y^{(a-1)} \frac{q^{3-3a} K^{-2} - q^{1-a} }{q^2-1}  \LR{\kk;2-m}{c}  F^{(2m-1-2c-a)} \right),
\notag
\\
\texttt{A}_1
&= \sum_{c=0}^{m-1} \sum_{a=0}^{2m-1-2c} q^{\binom{2c}{2} -a(2m-1-2c-a)-2a} 
\Y^{(a)}   \LR{\kk;2-m}{c} K^{-1} F^{(2m-1-2c-a)}.
\end{align*}

Using \eqref{sYa} we can also rewrite $t \cdot \texttt{S}_1$ in the $\Y \kk F$ form as 
\[
t \cdot \texttt{S}_1 =\texttt{B}_1 +\texttt{B}_0, 
\]
where
\begin{align*} 
\texttt{B}_1
&=  \sum_{c=0}^{m-1} \sum_{a=0}^{2m-2-2c} q^{\binom{2c+1}{2}-2m+2 -a(2m-2-2c-a)}
\cdot  \\
& \qquad \quad \left( [a+1] \Y^{(a+1)}  \LR{\kk;2-m}{c} K^{-1}  F^{(2m-2-2c-a)}
\right.
\\
& \qquad \qquad +q^{-2a-2} [2m-1-2c-a] \Y^{(a)}  \LR{\kk;1-m}{c} K^{-1}  F^{(2m-1-2c-a)} 
\\
& \qquad \qquad +
\left. \Y^{(a-1)} \frac{q^{3-3a} K^{-2} - q^{1-a} }{q^2-1} \LR{\kk;2-m}{c} K^{-1}  F^{(2m-2-2c-a)} \right),
\\
\texttt{B}_0
&= \sum_{c=0}^{m-1} \sum_{a=0}^{2m-2-2c} q^{\binom{2c+1}{2}-2m+2 -a(2m-2-2c-a) -2a} 
\Y^{(a)}  \LR{\kk;2-m}{c} K^{-2}  F^{(2m-2-2c-a)}.
\end{align*}

Hence by \eqref{eq:ttodd2:dvd} we have 
\begin{equation}   \label{eq:dd2}
[2m] \dvd{2m} = t \cdot \dvd{2m-1} =(\texttt{A}_0 +\texttt{B}_0) +(\texttt{A}_1+\texttt{B}_1).
\end{equation}
We shall rewrite $\texttt{A}_0+\texttt{B}_0$ in the form
\[
\texttt{A}_0 +\texttt{B}_0 =  \sum_{c=0}^m \sum_{a=0}^{2m-2c} \Y^{(a)} \texttt{f}_{a,c}^0(\kk) F^{(2m-2c-a)}, 
\]
where
\begin{align*}
\texttt{f}_{a,c}^0(\kk) &= q^{\binom{2c}{2} -(a-1)(2m-2c-a)}  [a]\,  \LR{\kk;2-m}{c}  \\
&\quad + \left( q^{\binom{2c}{2} -a(2m-1-2c-a)-2a} [2m-2c-a]\, \LR{\kk;1-m}{c} \right. \\
&\qquad\quad \left. +q^{\binom{2c-2}{2} -(a+1)(2m-2c-a)}  \frac{q^{-3a} K^{-2} - q^{-a} }{q^2-1} \LR{\kk;2-m}{c-1} \right.
\\
&\qquad\quad \left. + q^{\binom{2c-1}{2}-2m+2 -a(2m-2c-a) -2a} 
\Y^{(a)}  \LR{\kk;2-m}{c-1} K^{-2} \right). 
\end{align*}

A direct computation shows that
\begin{align*}
& \texttt{f}_{a,c}^0(\kk) 
\\
&= q^{\binom{2c-1}{2}-1 -a(2m-2c-a)} \cdot q^{2m-a} [a] \,  \LR{\kk;2-m}{c} \\
&\qquad + q^{\binom{2c-1}{2}-1 -a(2m-2c-a)} \LR{\kk;2-m}{c-1}  \cdot
\\
&\qquad
 \left( q^{2c-a} [2m-2c-a] \frac{q^{4-4m}K^{-2}-q^2}{q^{4c}-1}
 +q^{a+3-2m}  \frac{q^{-3a} K^{-2} - q^{-a} }{q^2-1}  + q^{3-2m -2a} K^{-2}
 \right)  
\\%
&= q^{\binom{2c-1}{2}-1 -a(2m-2c-a)} \left( q^{2m-a} [a] \,  \LR{\kk;2-m}{c}  + \LR{\kk;2-m}{c-1} q^{-a} [2m-a]  \frac{q^{4c+4-4m}K^{-2}-q^2}{q^{4c}-1} \right)
\\%
&= q^{\binom{2c-1}{2}-1 -a(2m-2c-a)} [2m] \LR{\kk;2-m}{c}.
\end{align*}

On the other hand, we shall rewrite $\texttt{A}_1+\texttt{B}_1$ in the form
\[
\texttt{A}_1 +\texttt{B}_1 =  \sum_{c=0}^{m-1} \sum_{a=0}^{2m-1-2c} \Y^{(a)} \texttt{f}_{a,c}^1(\kk) K^{-1} F^{(2m-1-2c-a)}, 
\]
where
\begin{align*}
\texttt{f}_{a,c}^1(\kk) &= q^{\binom{2c}{2} -a(2m-1-2c-a)-2a}    \LR{\kk;2-m}{c} 
+  q^{\binom{2c+1}{2}-2m+2 -(a-1)(2m-1-2c-a)} [a]\, \LR{\kk;2-m}{c}  \\
&\qquad\quad  + q^{\binom{2c+1}{2}-2m+2 -a(2m-2-2c-a) -2a-2} [2m-1-2c-a] \LR{\kk;1-m}{c}
\\
&\qquad\quad +q^{\binom{2c-1}{2}-2m+2 -(a+1)(2m-1-2c-a)} \frac{q^{-3a} K^{-2} - q^{-a} }{q^2-1} \LR{\kk;2-m}{c-1}
\\
&=: \texttt{U}_1 +\texttt{U}_2.
\end{align*}
In the above we have denoted by $\texttt{U}_1$  the sum of the first two summands 
and by $\texttt{U}_2$ the sum of the last 2 summands of $\texttt{f}_{a,c}^1(\kk)$.
A simple computation shows that
\begin{align*}
\texttt{U}_1
 &= q^{\binom{2c}{2}+1-2m  -a(2m-1-2c-a)} \cdot   (q^{2m-2a-1} + q^{2m-a}[a])  \LR{\kk;2-m}{c}
 \\
 &= q^{\binom{2c}{2}+1-2m  -a(2m-1-2c-a)} \cdot   q^{2m-a-1}[a+1]  \LR{\kk;2-m}{c}.
\end{align*}
Moreover, a direct computation shows that $\texttt{U}_2$ is equal to
\begin{align*}
&=  q^{\binom{2c}{2}+1-2m  -a(2m-1-2c-a)} \LR{\kk;2-m}{c-1}  \cdot
\\
& \qquad 
\left(
q^{2c-a-1} [2m-1-2c-a] \frac{q^{4-4m}K^{-2}-q^2}{q^{4c}-1}
+ q^{3-2m+a} \frac{q^{-3a} K^{-2} - q^{-a} }{q^2-1} 
\right)
\\
& =q^{\binom{2c}{2}+1-2m  -a(2m-1-2c-a)}  \LR{\kk;2-m}{c-1} \cdot 
\left(
q^{-a-1} [2m-1-a] \frac{q^{4c+4-4m}K^{-2}-q^2}{q^{4c}-1}
\right)
\\
&= q^{\binom{2c+1}{2}-2m -a(2m-2c-a)} \cdot q^{-a-1} [2m-a-1]  \LR{\kk;2-m}{c}.
\end{align*}
Hence we conclude that 
\[
\texttt{f}_{a,c}^1(\kk) =  \texttt{U}_1 +\texttt{U}_2 
=q^{\binom{2c}{2}+1-2m  -a(2m-1-2c-a)} [2m]  \LR{\kk;2-m}{c}.
\]
The formula \eqref{dvd2m} follows from \eqref{eq:dd2} and
the formulae of $\texttt{f}_{a,c}^0(\kk)$ and $\texttt{f}_{a,c}^1(\kk)$ above.

\vspace{.3cm}

(2) Now we shall prove the formula \eqref{dvd2m+1}  for $\dvd{2m+1}$ 
by assuming the formula \eqref{dvd2m} for $\dvd{2m}$.

Let us denote the 2 summands for $\dvd{2m}$ in \eqref{dvd2m} by $\texttt{T}_0, \texttt{T}_1$, and so 
\[
\dvd{2m} =\texttt{T}_0 +\texttt{T}_1.
\]
Using \eqref{sYa} we can rewrite $t \cdot \texttt{T}_0$ in the $\Y \kk F$ form as 
\[
t \cdot \texttt{T}_0 =\texttt{C}_0 +\texttt{C}_1, 
\]
where
\begin{align*}   
\texttt{C}_0
&=  \sum_{c=0}^m \sum_{a=0}^{2m-2c} q^{\binom{2c-1}{2}-1 -a(2m-2c-a)}
\cdot 
\notag \\
& \left( [a+1] \Y^{(a+1)}   \LR{\kk;2-m}{c}  F^{(2m-2c-a)}
 +q^{-2a} [2m+1-2c-a] \Y^{(a)}   \LR{\kk;1-m}{c}  F^{(2m+1-2c-a)} \right.
 \notag \\
& \qquad +
\left. \Y^{(a-1)} \frac{q^{3-3a} K^{-2} - q^{1-a} }{q^2-1}  \LR{\kk;2-m}{c}  F^{(2m-2c-a)} \right),
\notag
\\
\texttt{C}_1
&= \sum_{c=0}^m \sum_{a=0}^{2m-2c} q^{\binom{2c-1}{2}-1 -a(2m-2c-a)-2a} 
\Y^{(a)}  \LR{\kk;2-m}{c}  K^{-1} F^{(2m-2c-a)}.
\end{align*}

Using \eqref{sYa} we can also rewrite $t \cdot \texttt{T}_1$ in the $\Y \kk F$ form as 
\[
t \cdot \texttt{T}_1 =\texttt{D}_1 +\texttt{D}_0, 
\]
where
\begin{align*} 
\texttt{D}_1
&=  \sum_{c=0}^{m-1} \sum_{a=0}^{2m-1-2c} q^{\binom{2c}{2}+1-2m  -a(2m-1-2c-a)}
\cdot 
\notag \\
&\qquad  \left( [a+1] \Y^{(a+1)}  \LR{\kk;2-m}{c} K^{-1}  F^{(2m-1-2c-a)} \right.
\\
&\qquad\qquad +q^{-2a-2} [2m-2c-a] \Y^{(a)}   \LR{\kk;1-m}{c} K^{-1}  F^{(2m-2c-a)} 
 \notag \\
& \qquad\qquad 
 + \left. \Y^{(a-1)} \frac{q^{3-3a} K^{-2} - q^{1-a} }{q^2-1}  \LR{\kk;2-m}{c} K^{-1}  F^{(2m-1-2c-a)} \right),
\notag
\\
\texttt{D}_0
&= \sum_{c=0}^{m-1} \sum_{a=0}^{2m-1-2c} q^{\binom{2c}{2}+1-2m  -a(2m-1-2c-a) -2a} 
\Y^{(a)}  \LR{\kk;2-m}{c} K^{-2}  F^{(2m-1-2c-a)}.
\end{align*}

We denote the two summands in \eqref{dvd2m-1} by $\texttt{G}_0$ and $\texttt{G}_1$, so that 
\[
\dvd{2m-1} =\texttt{G}_0 +\texttt{G}_1.
\] 
Hence by \eqref{eq:ttodd2:dvd} we have 
\begin{align}
  \label{eq:CDG01K}
\begin{split}
[2m+1] \dvd{2m+1} &= t \cdot \dvd{2m} -[2m] \dvd{2m-1} 
\\
&=(\texttt{C}_0 +\texttt{D}_0 -[2m]\texttt{G}_0) +(\texttt{C}_1+\texttt{D}_1 -[2m]\texttt{G}_1).
\end{split}
\end{align} 
We shall write $\texttt{C}_0+\texttt{D}_0 -[2m]\texttt{G}_0$ in the form
\[
\texttt{C}_0 +\texttt{D}_0 -[2m]\texttt{G}_0 =  \sum_{c=0}^{m} \sum_{a=0}^{2m+1-2c} \Y^{(a)} \texttt{g}_{a,c}^0(\kk) F^{(2m+1-2c-a)}.
\]
Indeed $\texttt{g}_{a,c}^0 (\kk)$ can be organized (by separating the second summand in $\texttt{C}_0$) as
\begin{align}
  \label{gac0K}
\texttt{g}_{a,c}^0 (\kk) &= q^{\binom{2c-1}{2}-1 -a(2m-2c-a) -2a}  [2m+1-2c-a]\, \LR{\kk;1-m}{c}  +    \texttt{X}_0, 
\end{align}
where  
\begin{align*}
\texttt{X}_0 &= q^{\binom{2c-1}{2}-1 -(a-1)(2m+1-2c-a)}   [a]\,  \LR{\kk;2-m}{c} \\
&\quad   +q^{\binom{2c-3}{2}-1 -(a+1)(2m+1-2c-a)}  \frac{q^{-3a} K^{-2} - q^{-a} }{q^2-1} \LR{\kk;2-m}{c-1} 
\\
&\quad +  q^{\binom{2c-2}{2}+1-2m  -a(2m+1-2c-a) -2a}  \LR{\kk;2-m}{c-1}  K^{-2} \\
&\quad   - q^{\binom{2c-2}{2} -a(2m+1-2c-a)}  [2m] \LR{\kk;2-m}{c-1}.
\end{align*}
We  rewrite $\texttt{X}_0 = q^{\binom{2c}{2} -a(2m+1-2c-a)} \LR{\kk;2-m}{c-1} \texttt{Z}_0$,
where
\begin{align*}
\texttt{Z}_0 &= q^{2m-a-4c+1} [a] \frac{q^{4c+4-4m} K^{-2} - q^2}{q^{4c}-1} 
+q^{4-2m+a-4c} \frac{q^{-3a} K^{-2} - q^{-a} }{q^2-1} 
\\
&\quad + q^{4-2m-4c-2a} K^{-2} 
-q^{3-4c} [2m]. 
\end{align*}
A direct computation shows that 
\[
\texttt{Z}_0 =q^{2m+1-2c-a} [2c+a] \frac{q^{4-4m} K^{-2} - q^2}{q^{4c}-1},
\]
and this give us
\begin{align*}
\texttt{X}_0 =q^{\binom{2c}{2} -a(2m+1-2c-a)} q^{2m+1-2c-a}  [2c+a]\, \LR{\kk;1-m}{c}.
\end{align*}
Plugging this formula for $\texttt{X}_0$ into \eqref{gac0K} we obtain
\begin{align}
\label{gac0K:answer}
\texttt{g}_{a,c}^0 (\kk) 
%&= q^{\binom{2c}{2} -a(2m+1-2c-a)} \left(q^{-2c-a} [2m+1-2c-a] +q^{2m+1-2c-a}  [2c+a] \right) \LR{\kk;1-m}{c} \\
&= q^{\binom{2c}{2} -a(2m+1-2c-a)} [2m+1] \LR{\kk;1-m}{c}.
\end{align}

Now let us simplify  $\texttt{C}_1+\texttt{D}_1 -[2m]\texttt{G}_1$ by writing it in the form
\[
\texttt{C}_1 +\texttt{D}_1 -[2m]\texttt{G}_1 =  \sum_{c=0}^{m} \sum_{a=0}^{2m-2c} \Y^{(a)} \texttt{g}_{a,c}^1(\kk) K^{-1} F^{(2m-2c-a)}.
\]
Indeed $\texttt{g}_{a,c}^1 (\kk)$ can be organized (by pulling out the second summand of $\texttt{D}_1$) as
\begin{align}
  \label{gac1K}
\texttt{g}_{a,c}^1 (\kk) &=  q^{\binom{2c}{2}+1-2m  -a(2m-1-2c-a) -2a-2} [2m-2c-a] \LR{\kk;1-m}{c}  + \texttt{X}_1, 
\end{align}
where  
\begin{align*}
\texttt{X}_1 &=
q^{\binom{2c-1}{2}-1 -a(2m-2c-a)-2a}  \LR{\kk;2-m}{c} 
\\
&\qquad + q^{\binom{2c}{2}+1-2m  -(a-1)(2m-2c-a)} [a]   \LR{\kk;2-m}{c}   
\\
& \qquad + q^{\binom{2c-2}{2}+1-2m  -(a+1)(2m-2c-a)}   \frac{q^{-3a} K^{-2} - q^{-a} }{q^2-1} \LR{\kk;2-m}{c-1}
 \\
&\quad   - q^{\binom{2c-1}{2}-2m+2 -a(2m-2c-a)} [2m] \LR{\kk;2-m}{c-1}.
\end{align*}

We  rewrite 
\[
\texttt{X}_1 = q^{\binom{2c+1}{2}-2m -a(2m-2c-a)}  \LR{\kk;2-m}{c-1} \texttt{Z}_1,
\]
where
\begin{align*}
\texttt{Z}_1 &= q^{2m-2a-4c} \frac{q^{4c+4-4m} K^{-2} - q^2}{q^{4c}-1} 
+q^{1+2m-a-4c} [a] \frac{q^{4c+4-4m} K^{-2} - q^2}{q^{4c}-1} 
\\
&\qquad+q^{4-2m+a-4c} \frac{q^{-3a} K^{-2} - q^{-a} }{q^2-1} 
 -q^{3-4c} [2m]. 
\end{align*}

A direct computation shows that
\begin{align*}
\texttt{Z}_1 
%&= q^{2m-a-4c}[a+1] \frac{q^{4c+4-4m} K^{-2} - q^2}{q^{4c}-1} 
%+q^{4-2m+a-4c} \frac{q^{-3a} K^{-2} - q^{-a} }{q^2-1}  -q^{3-4c} [2m] \\
&= q^{2m-2c-a} [2c+a+1] \frac{q^{4-4m} K^{-2} -q^2}{q^{4c}-1},
\end{align*}
and this gives us
\begin{align*}
\texttt{X}_1 = q^{\binom{2c+1}{2}-2m -a(2m-2c-a)}  q^{2m-2c-a} [2c+a+1] \,  \LR{\kk;1-m}{c}.
\end{align*}
Plugging this formula for $\texttt{X}_1$ into \eqref{gac1K} we obtain
\begin{align}
\label{gac1K:answer}
\texttt{g}_{a,c}^1 (\kk) & =q^{\binom{2c+1}{2}-2m -a(2m-2c-a)} [2m+1]  \LR{\kk;1-m}{c}.
\end{align}

The formula for $\dvd{2m+1}$ now follows from \eqref{eq:CDG01K}, \eqref{gac0K:answer} and \eqref{gac1K:answer}.

This completes the proof of Theorem~\ref{thm:iDPK:odd}.

%%%%%%%%%%%%%%


\begin{thebibliography}{DD17}\frenchspacing

%\bibitem[Bao16]{Bao16}  H.~Bao, 
%{\em Kazhdan-Lusztig theory of super type $D$ and quantum symmetric pairs}, Preprint, 
%\href{http://arxiv.org/abs/1603.05105}{arXiv:1603.05105}.

\bibitem[BW13]{BW13} H. Bao and W. Wang,
{\em  A new approach to Kazhdan-Lusztig theory  of type $B$ via quantum symmetric pairs},   
 \href{http://arxiv.org/abs/1310.0103}{arXiv:1310.0103}v2.

\bibitem[BW16]{BW16} H. Bao and W. Wang,
{\em  Canonical bases arising from quantum symmetric pairs},   
 \href{http://arxiv.org/abs/1610.09271}{arXiv:1610.09271}.

\bibitem[K93]{K93} 
T. Koornwinder,
{\em  Askey-Wilson polynomials as zonal spherical functions on the $SU(2)$ quantum group},
SIAM J. Math. Anal. {\bf 24} (1993), 795--813.

\bibitem[Le99]{Le99} G. Letzter, 
{\em Symmetric pairs for quantized enveloping algebras}, J. Algebra {\bf 220}  (1999), 729--767.

\bibitem[LW15]{LW15}   Y.~Li and W.~Wang,
{\em Positivity vs negativity of canonical bases}, Proceedings for Lusztig's 70th birthday conference,
Bulletin of  Institute of Mathematics Academia Sinica (N.S.), to appear, \href{http://arxiv.org/abs/1501.00688}{arXiv:1501.00688}v4.

\bibitem[L93]{L93} G. Lusztig, {\em Introduction to quantum groups}, 
Modern Birkh\"auser Classics, Reprint of the 1993 Edition,
Birkh\"auser, Boston, 2010.

\end{thebibliography}
\end{document}